\documentclass[12pt]{amsart}

\usepackage{amssymb}
\usepackage{latexsym}
\usepackage{exscale}
\usepackage{lscape}
\usepackage{amsmath}

\usepackage[colorlinks=true, pdfstartview=FitV, linkcolor=blue, citecolor=blue, urlcolor=blue]{hyperref}

\newtheorem{theorem}{Theorem}[section]
\newtheorem{lemma}[theorem]{Lemma}
\newtheorem{cor}[theorem]{Corollary}
\newtheorem{prop}[theorem]{Proposition}

\theoremstyle{definition}
\newtheorem{definition}[theorem]{Definition}

\theoremstyle{remark}
\newtheorem{remark}[theorem]{Remark}

\allowdisplaybreaks

\numberwithin{equation}{section}

\newcommand{\R}{\mathbb R}

\newcommand{\Z}{\mathbb Z}

\DeclareMathOperator{\supp}{supp}

\DeclareMathOperator{\diam}{diam}
\DeclareMathOperator*{\essinf}{ess\,inf}
\DeclareMathOperator*{\esssup}{ess\,sup}

\newcommand{\pp}{{p(\cdot)}}

\newcommand{\Lp}{L^{p(\cdot)}}

\newcommand{\qq}{{q(\cdot)}}

\newcommand{\D}{\mathcal{D}}

\def\Xint#1{\mathchoice
   {\XXint\displaystyle\textstyle{#1}}%
   {\XXint\textstyle\scriptstyle{#1}}%
   {\XXint\scriptstyle\scriptscriptstyle{#1}}%
   {\XXint\scriptscriptstyle\scriptscriptstyle{#1}}%
   \!\int}
\def\XXint#1#2#3{{\setbox0=\hbox{$#1{#2#3}{\int}$}
     \vcenter{\hbox{$#2#3$}}\kern-.5\wd0}}

\def\avgint{\Xint-}

\begin{document}

\date{November 5, 2015}

\author{D. Cruz-Uribe, OFS}
\address{Dept. of Mathematics \\ University of Alabama \\
Tuscaloosa, AL 35487}
\email{dcruzuribe@ua.edu}

\author{P. Shukla}
\email{shukla.parantap@gmail.com}

\title[Boundedness of fractional maximal operators]
{The boundedness of fractional maximal operators on
variable Lebesgue spaces over spaces of homogeneous type}

\thanks{The first author is supported by NSF Grant 1362425 and
  research funds from the Dean of the College of Arts \& Sciences,
  University of Alabama.  Earlier parts of the work were supported by
  the Stewart-Dorwart research fund at Trinity College.  The authors
  would like to thank Diego Maldonado for several insightful
  conversations and suggestions.}

\begin{abstract}\label{abstract}
Given a space of homogeneous type $(X,d,\mu)$, we give sufficient
conditions on a variable exponent $\pp$ so that  the
fractional maximal operator $M_\eta$, $0\leq \eta <1$,   maps $\Lp(X)$
to $L^\qq(X)$, where $1/\pp-1/\qq=\eta$.  In the endpoint case we
also prove the corresponding weak type inequality.  As an application we prove
norm inequalities for the fractional integral operator $I_\eta$.  Our
proof for the fractional maximal operator uses the theory of dyadic
cubes on spaces of homogeneous type, and even in the Euclidean setting
it is simpler than existing proofs.  For the fractional integral
operator we extend a pointwise inequality of Welland to spaces of
homogeneous type.  Our work  generalizes results in
\cite{MR2414490,MR2493649} from the Euclidean case and extends recent
work by Adamowicz, {\em et al.}~\cite{MR3322604} on the
Hardy-Littlewood maximal operator on spaces of homogeneous type.
\end{abstract}

\maketitle

\section{Introduction}\label{introduction}

In this paper we study the boundedness of the fractional maximal
operator on variable Lebesgue spaces defined over spaces $(X,d,\mu)$  of
homogeneous type. Variable Lebesgue spaces are Banach
function spaces which generalize the classical Lebesgue spaces;
intuitively, they consist of all measurable functions that satisfy
\[ \int_X |f(x)|^{p(x)}\,d\mu < \infty.  \]
These spaces have been intensively studied for the past twenty
years: see the books~\cite{MR3026953,DHHR} for detailed histories and
references. Such spaces were first studied on $\R^n$ equipped with the
standard Euclidean distance and Lebesgue measure.  However, more
recently there has been interest in working with variable exponent
spaces defined over spaces of homogeneous type.  See, for
instance,~\cite{MR2516251,MR2248828,
  MR2738962,MR2296640,MR3100969,MR2387413}. 

A central problem is to determine sufficient conditions on the exponent
$\pp$ such that the Hardy-Littlewood and fractional maximal operators are
bounded on $\Lp(X)$.  For  a detailed history of this problem in the
Euclidean setting, see~\cite{MR3026953}. This problem 
was first considered on metric measure spaces in~\cite{HHP}.
However, in this and subsequent papers, results were proved with
the (somewhat) unnatural restrictions that either $X$ is bounded or
that the exponent $\pp$ is constant outside a large
ball. Adam\-o\-wicz, Harjulehto and H\"ast\"o~\cite{MR3322604} were
able to remove this restriction and established the boundedness of the 
maximal operator on $\Lp(X)$ on an unbounded quasi-metric space $X$. While 
they did not assume that the underlying measure is doubling, they did assume
that $1/\pp$ satisfies a Diening type condition.
  

In our paper we consider the same problem for the fractional maximal
operator:  given $0\leq \eta <1$, define the operator $M_\eta$ by
\[ M_\eta f(x) = \sup_{B\ni x} \mu(B)^\eta \avgint_B |f|\,d\mu.  \]
When $\eta=0$ this reduces to the Hardy-Littlewood maximal operator
and we write $M$ instead of $M_0$. On classical Lebesgue spaces
over $\R^n$, norm inequalities for $M_\eta$ are well
known.  For variable Lebesgue spaces in the Euclidean case, norm
inequalities for $M_\eta$ were first proved in~\cite{MR2414490} and then
subsequently in \cite{MR2493649}.
Our main results  are  generalizations of the results in \cite{MR2414490,MR2493649} to the setting of spaces
of homogeneous type.   We state them here, though we defer the statement
of precise
definitions for all of our results until Section~\ref{prelim}.  

\begin{theorem}\label{main theorem 1}
Let $(X,d,\mu)$ be a space of homogeneous type.  Given $0\leq \eta<1$, 
let $\pp : X \rightarrow [1,\infty]$ be such that $1/\pp\in LH$ and 
$1<p_-\leq p_+\leq1/\eta$. For each $x\in X$, define $\qq$ 
pointwise by $1/p(x)-1/q(x)=\eta$. 
Then there exists a constant $C=C(\pp,\eta,X)$ such that for all $f
\in \Lp(X)$, 
\begin{equation}\label{main equation 1}
\|M_{\eta}f\|_{\qq}\leq C\|f\|_{\pp}.
\end{equation}
Moreover, if $\mu(X)<+\infty$, then we can replace the hypothesis
that $1/\pp\in LH$ with $1/\pp\in LH_0$.  
\end{theorem}


When $\eta=0$, Theorem~\ref{main theorem 1} is a particular case of
the results in~\cite{MR3322604}.  Unlike in this paper we prefer to
work on spaces of homogeneous type.  Further, we do not {\em
  a priori} assume that  
$1/\pp$ satisfies a Diening type condition: rather,  we derive this property as a consequence of 
the log-H\"older continuity of $1/\pp$. Our method of proof is very different from
theirs and generalizes an argument given in \cite{MR2493649} (see also
\cite{MR3026953}) that is 
based on the Calder\'on-Zygmund decomposition in Euclidean spaces.
In order to do so we exploit the theory of dyadic cubes on
spaces of homogeneous type, first introduced by
Christ~\cite{MR1096400} and refined by Hyt\"onen and
Kairema~\cite{MR2901199}.   We want to emphasize that our results are
not simply a translation of this earlier work to the setting of spaces
of homogeneous type:  even in the Euclidean case our proof is a
significant refinement of the one in~\cite{MR2493649}, and in
adapting it to spaces of homogeneous type we needed to overcome
various technical obstacles.

We also prove a weak type inequality in the endpoint case $p_- = 1$.  In
the Euclidean case this was first proved in~\cite{MR2414490}; in the
setting of spaces of homogeneous type it is new, even when $\eta=0$.  

\begin{theorem}\label{weak theorem 1}
  Let $(X,d,\mu)$ be a space of homogeneous type. Given
  $0 \leq \eta < 1$, let $\pp : X \rightarrow \left[1,\infty\right]$
  be such that $1/\pp \in LH$ and $1 = p_- \leq p_+ \leq 1/\eta$. Then
  there exists a positive constant $C=C(\pp,\eta,X)$ such that for all
  $f\in \Lp(X)$,
\begin{equation}\label{weak equation 1}
\underset{t > 0}{\sup}\,t\|\chi_{\{M_{\eta} f > t\}}\|_{\qq} \leq 
C\|f\|_{\pp},
\end{equation}
where $1/\pp - 1/\qq = \eta$. Furthermore, if $\mu(X) < +\infty$, we can replace the hypothesis
that $1/\pp \in LH$ with $1/\pp \in LH_0$.
\end{theorem}

As an application of Theorems~\ref{main theorem 1} and~\ref{weak theorem 1} 
we prove strong and weak type norm inequalities for the
fractional integral operator (also referred to as the Riesz potential),
\[ I_\eta f(x) = \int_X
\frac{f(y)}{\mu\big(B(x,d(x,y))\big)^{1-\eta}}\,d\mu(y).  \]
These operators have been extensively studied on spaces of homogeneous
type for constant exponents:
see~\cite{MR3058926} and the references it contains.  In the Euclidean
case, strong and weak type inequalities on variable Lebesgue spaces were proved
in~\cite{MR2414490} (but see the references there for earlier, partial
results).  On spaces of homogeneous type they were considered
in~\cite{MR2248828,MR2738962} when $\mu(X)<+\infty$.  For our results
we need to impose an additional condition on the space $(X,d,\mu)$; we
will discuss this hypothesis in more detail in Section~\ref{prelim}.

\begin{theorem}\label{frac int theorem 1}
Let $(X,d,\mu)$ be a reverse doubling space. Given $0 < \eta < 1$,
let $\pp :X\rightarrow [1,+\infty)$ be such that 
$\pp \in LH$ and $1 < p_- \leq p_+ <1/\eta$. Define $\qq$ by  $1/\pp -
1/\qq = \eta$.   Then there exists $C=C(\pp,\eta,X)$ such that for all $f\in \Lp(X)$, 
\[ \|I_\eta f\|_\qq \leq C\|f\|_\pp. \]
Moreover,  if $\mu(X) < +\infty$, we can replace the hypothesis
that $1/\pp \in LH$ with $1/\pp \in LH_0$.
\end{theorem}

\begin{theorem}\label{frac int theorem 2}
Let $(X,d,\mu)$ be a reverse doubling space. Given $0 < \eta < 1$,
let $\pp :X\rightarrow [1,+\infty)$ be such that 
$\pp \in LH$ and $1 = p_- \leq p_+ <1/\eta$. Define $\qq$ by  $1/\pp -
1/\qq = \eta$.   Then there exists $C=C(\pp,\eta,X)$ such that for all $f\in \Lp(X)$, 

\[
\underset{t > 0}{\sup}\,t\,\|\chi_{\{|I_{\eta} f| > t\}}\|_{\qq} \leq 
C\|f\|_{\pp}. 
\]
Moreover,  if $\mu(X) < +\infty$, we can replace the hypothesis
that $1/\pp \in LH$ with $1/\pp \in LH_0$.
\end{theorem}

\medskip

As an immediate consequence of Theorem~\ref{frac int theorem 1} we can
prove norm inequalities for other variants of the fractional integral
operator on an Ahlfors
regular space.  We say that $(\Omega,d,\mu)$ is Ahlfors regular if
there exist  constants $C_1,\,C_2,\,Q>0$ such that for every $x\in
\Omega$ and $r>0$
\begin{equation} \label{eqn:ahlfors}
C_1r^Q \leq \mu(B(x,r)) \leq C_2r^Q.  
\end{equation}
The constant $Q$ is referred to as the dimension of the space.

Given $0<\alpha<Q$, define the operators
\[  I_\alpha^* f(x) = \int_X \frac{f(y)}{d(x,y)^{Q-\alpha}}\,d\mu(y) \]
and
\[  I_\alpha^{**} f(x) = \int_X \frac{f(y)d(x,y)^\alpha}{\mu(B(x,d(x,y)))}\,d\mu(y). \]
These operators have also been extensively studied in spaces of
homogeneous type for  constant exponents:  see~\cite{MR3058926,sawyer-wheeden92}.   They have applications to the study of
Sobolev and Poincar\'e inequalities over metric spaces:
see~\cite{MR1336257,MR2569546}.

If $(\Omega,d,\mu)$ is Ahlfors regular, then it is immediate that these operators are pointwise equivalent to
$I_\eta f$ with $\eta=\alpha/Q$,  and strong and weak type norm
inequalities follow from Theorems~\ref{frac int theorem 1}
and~\ref{frac int theorem 2}.  For brevity we only state the
strong type inequality with the (implicit) assumption that $\mu(X)=+\infty$;
precise statements of the other results are left to the interested
reader.  

\begin{cor} \label{cor:variant-frac-int}
Suppose $(X,d,\mu)$ is an Ahlfors regular space with dimension $Q$. Given
$0 < \alpha<Q$, 
let $\pp :X\rightarrow [1,+\infty)$ be such that 
$\pp \in LH$ and $1 < p_- \leq p_+ <Q/\alpha$. Define $\qq$ by  $1/\pp -
1/\qq = \alpha/Q$.   Then there exists $C=C(\pp,Q,\alpha,X)$ such that for all $f\in \Lp(X)$, 
\[ \|I_\alpha^* f\|_\qq \leq C\|f\|_\pp. \]
The same inequality also holds for $I_\alpha^{**}$.
\end{cor}

\begin{remark} 
We can actually prove inequalities for $I_\alpha^*$ or
$I_\alpha^{**}$ assuming that the space is either upper or lower Ahlfors
regular--i.e., that either the righthand or lefthand inequality
in~\eqref{eqn:ahlfors} holds.  Details are left to the interested
reader.
\end{remark}

\medskip

The remainder of this paper is organized as follows.  In
Section~\ref{prelim} we gather together the necessary definitions and
a number of preliminary results about spaces of homogeneous type,
fractional maximal operators and variable Lebesgue spaces.  In
Section~\ref{main} we prove Theorem~\ref{main theorem 1}, and in
Section~\ref{section:weak-type} we prove Theorem~\ref{weak theorem 1}.
Since there are many similarities between the proofs of Theorem~\ref{main theorem 1} and 
Theorem~\ref{weak theorem 1}, we omit overlapping details.
Finally, in Section~\ref{section:fractional} we prove Theorems~\ref{frac int theorem 1} 
and~\ref{frac int theorem 2}.  
Our proof involves extending a pointwise estimate due to Welland \cite{MR0369641} to 
spaces of homogeneous type that relates the fractional integral $I_\eta$ to the 
fractional maximal operator $M_\eta$.
This estimate  is interesting in its own right and should have applications to other
problems on spaces of homogeneous type.  (A related estimate was proved
in~\cite{garcia-cuerva-martell01} where it was used to derive weighted norm
inequalities.)

Throughout this paper our notation is standard and will be defined as
needed.  Hereafter, $C$, $c$  will denote constants whose values
depend only on ``universal'' parameters and whose value may change
from line to line.  In particular, constants may depend on the
underlying triple $(X,d,\mu)$.    We will use the convention that
$1/\infty=0$ and $1/0=\infty$. 

\section{Preliminary results}\label{prelim}

In this section we gather some definitions and 
preliminary results.

\subsection*{Spaces of homogeneous type}
We begin with a definition.  For more information, 
see~\cite{MR2867756, MR1800917}.

\begin{definition}\label{prelim definition 1}
Given a set $X$ and a function
$d\,:\,X \times X\,\rightarrow\,[0,\infty)$, we say that $(X,d)$ is a
quasi-metric space if $d$ satisfies the 
following conditions:
\begin{enumerate}
\item $d(x,y)=0$ if and only if $x=y$; 
\item $d(x,y)=d(y,x)$ for all $x,y\in X$;
\item $d(x,y)\leq A_0 (d(x,z)+d(y,z))$ for all $x,y,z \in X$ 
 and some constant $A_0\geq 1$.
 \end{enumerate}
Property $(3)$ is called the quasi-triangle inequality and 
$A_0$ the quasi-metric constant. 
\end{definition}

\begin{definition}\label{prelim definition 2}
Given a quasi-metric space $(X,d)$ and a positive 
measure $\mu$ that is defined on the $\sigma$-algebra
generated by quasi-metric balls and open sets, 
we say that $(X,d,\mu)$ is a space of homogeneous type
if there exists a constant 
$C_\mu\geq 1$ such that for any $x\in X$ and any $r>0$,
\[
\mu(B(x,2r))\leq C_\mu\,\, \mu(B(x,r)),
\]
where $B(x,r)$ is the ball centred at $x$ with radius $r$. To avoid trivial measures we will always assume that 
$0\,<\,\mu(B)\,<\,+\infty$ for every ball $B$.
\end{definition}

A measure $\mu$ that satisfies the property in Definition \ref{prelim
  definition 2} is called doubling.  The next lemma gives a
consequence of this property referred to as the lower mass bound.  
The proof is well-known and we omit it. 

\begin{lemma}\label{prelim lemma 1}
Let $(X,d,\mu)$ be a space of homogeneous type. Then there exists a
positive constant $C=C(C_\mu,A_0)$ such that for all  $x\in X$,   $0<r<R$ and $y\in B(x,R)$, 
\begin{equation}\label{prelim equation 1}
\frac{\mu(B(y,r))}{\mu(B(x,R))}\geq\,
C\left(\frac{r}{R}\right)^{\log_2\,C_\mu}.
\end{equation}
\end{lemma}




The next lemma characterizes spaces of homogeneous type with
finite measure. We refer the reader to 
\cite{MR1430157} for a proof.

\begin{lemma}\label{prelim finite measure}
Let $(X,d,\mu)$ be a space of homogeneous type. Then $\mu(X) < +\infty$ if and only if 
$$\diam(X) = \sup_{x,y\in X} d(x,y) < +\infty.$$
\end{lemma}

\medskip

On Euclidean spaces dyadic cubes play a
fundamental role in harmonic analysis. In particular they let us 
define dyadic versions of various operators. 
Christ~\cite{MR1096400} constructed a system
of sets on a space of homogeneous type which satisfy many of the
essential properties of a system of dyadic cubes in Euclidean
space. His construction was further refined by Hyt{\"o}nen
and Kairema~\cite{MR2901199} and an equivalent formulation was 
given in \cite{MR3302574}. We will use the version from \cite{MR3302574}.  


\begin{theorem}\label{prelim theorem 1}
Let $(X,d,\mu)$ be a space of homogeneous type. There exist constants $C>0$, 
$0<\delta,\,\epsilon<1$ which depend on $X$, a family of sets 
$\D\,=\,\bigcup_{k\in \Z}\,\D_{k}$, and 
a collection of points $\left\{x_c(Q)\right\}_{Q\in \D}$ that satisfy the 
following properties: 

\begin{enumerate}
\item For every $k\in \Z$ the cubes in $\D_{k}$ are pairwise disjoint and
                        $X=\underset{Q\in \D_{k}}{\bigcup}Q$. We
                        will refer to the cubes in $\D_k$ as cubes in
                        the $k$-th generation;

\item If $Q_1,Q_2\in\D$, then either $Q_1\bigcap Q_2= \emptyset$,
                         $Q_1\subseteq Q_2$ or $Q_2\subseteq Q_1$;
\item For any $Q_1\in\D_k$ there exists at least one $Q_2\in\D_{k+1}$,
                          which is called a child of $Q_1$, such that $Q_2\subseteq Q_1$ and 
                          there exists exactly one $Q_3\in\D_{k-1}$, which is called a 
                          parent of $Q_1$, such that $Q_1\subseteq Q_3$;
\item If $Q_2$ is a child of $Q_1$, then $\mu(Q_2)\geq\epsilon\mu(Q_1)$;
\item For every $k$ and $Q\in \D_k$,   $B(x_c(Q),\delta^{k})\subseteq Q\subseteq B(x_c(Q),C\delta^{k})$.
\end{enumerate}
\end{theorem}

The collection $\D$ is referred to as a dyadic grid on $X$ and the
sets $Q\in \D$ as dyadic cubes. The last property in Theorem
\ref{prelim theorem 1} permits a comparison between a dyadic cube and
quasi-metric balls; however, we will also need a way to compare a
quasi-metric ball with dyadic cubes. For this reason it is important
to have a finite family of dyadic grids such that an arbitrary
quasi-metric ball is contained in a dyadic cube from one of these
grids. Such a finite family of dyadic grids is referred to as an
adjacent system of dyadic grids.

\begin{theorem}\label{prelim theorem 2}
  Let $(X,d,\mu)$ be a space of homogeneous type. There exists a
  positive integer $K=K(X)$, a finite constant $C=C(X)$, and a finite collection of
  dyadic grids, $\D^{t}$, $1\leq t \leq K$, such that given any
  ball $B=B(x,r)\subseteq X$ there exists $t$ and a dyadic cube
  $Q\in \D^{t}$ such that $B\subseteq Q$ and $\diam{Q}\leq Cr$. 
\end{theorem}

\subsection*{Reverse doubling and Ahlfors regular spaces}
For our results on fractional integral operators we need to impose an
additional condition on our underlying space.

\begin{definition} \label{defn:reverse-doubling}
Given a space of homogeneous type $(X,d,\mu)$, we say that it is a
reverse doubling space if there exists a constant $0<\gamma<1$ such
that for every $x\in X$ and $r>0$ such that $B(x,r)\subsetneq X$, 
\[ \mu \big(B(x,r/2)\big) \leq \gamma \mu(B(x,r)). \]
If this condition holds we also say that the measure $\mu$ is reverse doubling.
\end{definition}

On Euclidean space any doubling measure is reverse doubling; the same
is true on any metric (as opposed to quasi-metric) space that is
connected:  see~\cite{MR2867756}.  More generally, it holds on any
space of homogeneous type that satisfies a non-empty annuli condition:
for a precise definition and proof, see~\cite{MR2385658}.  

Reverse doubling spaces do not have atoms: this is the content of the
next lemma.

\begin{lemma} \label{lemma:no-atoms}
If $(X,d,\mu)$ is a reverse doubling space, then for all $x\in X$,
$\mu(\{x\})=0$. 
\end{lemma}

\begin{proof}
If $\mu(X)= +\infty$ then for any $x$ by the definition of reverse doubling
\[ \mu(\{x\}) = \lim_{i \rightarrow \infty} \mu(B(x,2^{-i}))
\leq \lim_{i \rightarrow \infty} \gamma^i \mu(B(x,1)) = 0. \]
Now assume that $\mu(X) < +\infty$ and let $x\in X$. Choose 
\[
0 < r < \frac{\diam(X)}{8A_0}.
\]
By definition there exist points $y,z \in X$ such that $2^{-1}\diam(X) < d(z,y)$.
If both $y$ and $z$ belong to $B(x,r)$, then an application of the quasi-triangle inequality gives

\[
\frac{1}{2} \diam(X) < d(y,z) \leq
A_0(d(x,y) + d(x,z)) < 2 A_0 r < \frac{1}{4}\diam(X),
\]
which is a contradiction. Therefore,  $B(x,r)\subsetneq X$ and we may replace the balls $B(x,2^{-i})$ with the balls 
$B(x,2^{-i}r)$ and repeat the previous argument in order 
to get $\mu(\{x\}) = 0$.

\end{proof}

\begin{remark} \label{rem:equiv}
  Macias and Segovia showed that on any space of homogeneous type
  $(X,d,\mu)$ there exists an equivalent quasi-metric $\rho$ such that
  the quasi-metric balls with respect to $\rho$ are open. Therefore we
  could have assumed from the outset that our $\sigma$-algebra is the
  Borel algebra and that $\mu$ is a positive Borel measure which is
  doubling. The definition of the reverse doubling condition would
  need to be changed slightly: there exist constants $C>0$ and $0<\gamma<1$ such that
  for any ball $B(x,r)\subsetneq X$ and any $i\geq 1$
\[
\mu(B(x,2^{-i}r))\leq
C\gamma^{i}
\mu(B(x,r)).
\]
For further details on this perspective, see~\cite{MR3183648}.
The proofs we give Section~\ref{section:fractional} go through with
essentially no change using this definition of reverse doubling
and we leave the details to the interested reader. 
\end{remark}

\medskip

\begin{remark}
If the space $(X,d,\mu)$ is Ahlfors regular, then it is immediate that
the measure $\mu$ is doubling.  It need not be reverse doubling,
but it does satisfy the weaker condition in Remark~\ref{rem:equiv}.
Given this, the proof of Corollary~\ref{cor:variant-frac-int} is a
straightforward modification of the proof of Theorem~\ref{frac int
  theorem 1} and so is omitted.
\end{remark}

\subsection*{Fractional maximal operators}
 We begin by restating the definition given in the Introduction.
Given a set $E$, $\mu(E)>0$, we will use the notation
\[ \avgint_E f\,d\mu = \frac{1}{\mu(E)}\int_E f\,d\mu. \]

\begin{definition}\label{prelim definition 3}
  Given a space of homogeneous type $(X,d,\mu)$ and $0\leq \eta < 1$,
 define the fractional maximal operator of
  order $\eta$ acting on  $f\in L^1_{loc}(X)$ by 
\[ M_{\eta}f(x)=
\sup_{B\ni x}\mu(B)^{\eta}\avgint_{B}|f|\,d\mu, \]
where the supremum is taken over all balls which contain the point $x$.  
When $\eta=0$
we write $M$ instead of $M_0$.  
\end{definition}

\begin{definition}\label{prelim definition 4}
Given a space of homogeneous type $(X,d,\mu)$, a
dyadic grid $\D$ on $X$ and $0\leq \eta < 1$, the 
dyadic fractional maximal operator of order $\eta$ with 
respect to $\D$ is defined by
\[
M^{\D}_{\eta}f(x)=
\underset{Q\in \D}{\sup_{x\in Q}}\mu(Q)^{\eta}
\avgint_{Q}|f|\,d\mu,
\]
When $\eta=0$ we write $M^\D$ instead of $M_0^\D$.  
\end{definition}

\medskip

Theorem \ref{prelim theorem 2} yields a pointwise comparison between
the fractional maximal operator and the dyadic fractional maximal
operators associated with an adjacent system of dyadic cubes.  The
following result is proved in~\cite{MR2901199,MR3058926}.

\begin{prop}\label{prelim proposition 1}
Given a space of homogeneous type  $(X,d,\mu)$,  let
$\{\D^{t}\}$ be the adjacent dyadic system from Theorem \textup{\ref{prelim  theorem 2}}.  Fix $0\leq \eta <1$.  
Then there exists a constant $C=C(\eta)\geq 1$ such that for
all $f\in L^1_{loc}(X)$, 
\[
M^{\D^{t}}_{\eta}f(x)\leq M_{\eta}f(x)
\hspace{14pt} \mbox{and} \hspace{14pt}
M_{\eta}f(x)\leq C\underset{t=1}{\overset{K}{\sum}}
M^{\D^{t}}_{\eta}f(x).
\]
\end{prop}

We will need a variant of the classical Calderon-Zygmund
decomposition adapted to spaces of homogeneous type. The 
proof is essentially the same as in the Euclidean case and we refer the
reader to \cite{MR3302574} for further details.

\begin{lemma}\label{prelim lemma 3}
  Given a space of homogeneous type $(X,d,\mu)$  such that $\mu(X)=+\infty$, let $\D$ be a dyadic
  grid on $X$.  Fix $0\leq \eta < 1$. Let
  $f\in L^{1}_{loc}(X)$ be a function such that
  $\mu(Q)^\eta\avgint_{Q}|f|\,d\mu\rightarrow0$ as $\mu(Q)\rightarrow\infty$
  where $Q\in \D$. Then for each $\lambda>0$, there exists a set of
  pairwise disjoint dyadic cubes $\{Q_{j}\}$ and a constant $C=C(X,\D)>1$ such
  that
\[
\{x\in X:M^{\D}_{\eta}f(x)>\lambda\}=
\underset{j}{\bigcup}Q_{j},
\]
and 
\[
\lambda<
\mu(Q_{j})^{\eta}\avgint_{Q_{j}}|f|\,d\mu \leq C\lambda.
\]
If $\mu(X)<+\infty$, then the same conclusion holds for all
$\lambda>\mu(X)^\eta \avgint_X |f|\,d\mu$.
\end{lemma}

Finally, we need two results which were proved in \cite{MR2493649} in
the Euclidean case; the proofs in spaces of homogeneous type are
identical and so we omit them.  The first is a pointwise approximation
theorem.   

\begin{lemma}\label{prelim lemma 9}
  Given a space of homogeneous type $(X,d,\mu)$, let $\D$ be a dyadic
  grid on $X$.  Fix $0\leq \eta < 1$.  Let $f_N$ be a sequence of
  non-negative functions that increases pointwise a.e. to a function
  $f$. Then the functions $M^{\D}_{\eta}f_N$ increase to
  $M^{\D}_{\eta}f$ pointwise. The same is true if we replace
  $M^\D_\eta$ by $M_\eta$.
\end{lemma}

The second will let us compare the fractional maximal operator to the
Hardy-Littlewood maximal operator.  

\begin{lemma}\label{prelim lemma 10}
  Fix $0\leq \eta < 1$ and suppose $r$ and $s$ satisfy $1<r<1/\eta$
  and $1/r-1/s=\eta$. Then for every set $E$ of finite measure and for
  every non-negative function $f$,
\begin{equation}\label{prelim equation 12}
\mu(E)^{\eta}\avgint_{E}f\,d\mu\leq
\left(\int_Ef^{r}\,d\mu\right)^{\frac{1}{r}-\frac{1}{s}}
\left(\avgint_{E}f\,d\mu\right)^{\frac{r}{s}}.
\end{equation}
\end{lemma}

\begin{remark}\label{prelim remark 2}
Fix a dyadic grid $\D$ and $0<\eta<1$.  Given $f\in L^{r}(X)$,  where $r$ and $s$ are as in
Lemma \ref{prelim lemma 10}, let $x\in X$ and $Q\in \D$ be such that $x\in Q$. If we let $E=Q$ in inequality \eqref{prelim equation 12} and take 
the supremum over all such dyadic cubes, we get 
\begin{equation}\label{prelim equation 13}
M^{\D}_{\eta} f(x)^{s}\leq \|f\|_{r}^{s-r} M^{\D} f(x)^{r}.
\end{equation}
If we further assume that $f\in L^{\infty}(X)$,  then for every $x\in X$
\begin{equation}\label{prelim equation 14}
M^{\D}_{\eta} f(x)^{s}\leq
\|f\|_{r}^{s-r} \|f\|_{\infty}^{r}<+\infty.
\end{equation}

We can actually say more.  By Lemma~\ref{prelim lemma 9} and
Marcinkiewicz interpolation, a standard argument shows that 
$M^{\D} : L^1(X)\rightarrow L^{1,\infty}(X)$ and $M^{\D}$ is
bounded operator on  $L^{r}(X)$ when $r>1$.  In this latter case we
immediately have that
\[
\|M^{\D}_{\eta} f\|_{s}^{s}\leq
\|f\|_{r}^{s-r} \|M^{\D} f\|_{r}^{r}\leq
C \|f\|_{r}^{s}.
\]
In other words, $M^{\D}_{\eta} : L^{r}(X) \rightarrow L^{s}(X)$ is a
bounded operator. 

When $r = 1$ we have that $M^{\D}_{\eta} : L^{r}(X) \rightarrow
L^{s,\infty}(X)$.  (If $r>1$ this follows at once from Chebyshev's
inequality and the strong type inequality.)  If $\eta=0$, this was
noted above.  If $\eta>0$, then by Lemma \ref{prelim lemma 3} there
exist disjoint dyadic cubes $\{Q_j\}$ such that 
\[ 
\mu(\{M^{\D}_{\eta} f > t\})^{1-\eta} = \bigg(\underset{j}{\sum} \mu(Q_j)\bigg)^{1-\eta} \leq
\underset{j}{\sum} \mu(Q_j)^{1-\eta},
\]
where
\[
\mu(Q_j)^{1-\eta} < \frac{1}{t} \int_{Q_j} |f|\,d\mu.
\]
Since $r = 1$ and $s = 1/(1-\eta)$ it follows that
\[
\sup_{t>0}t\|\chi_{\{M^{\D}_{\eta} f > t\}}\|_{s} \leq \|f\|_{r}.
\]
\end{remark}

\medskip

\subsection*{Variable Lebesgue spaces}
We now give the definition and some basic properties of variable
Lebesgue spaces.  For complete details see \cite{MR3026953,DHHR}.
Given a space of homogeneous type $(X,d,\mu)$, let
$\pp:X\rightarrow[1,\infty]$ be a measurable function and 
define the set $\Omega_{\infty}^{\pp}=\{x\in X:p(x)=\infty\}$. The variable 
Lebesgue space $\Lp(X)$ is the set of measurable functions such that for some $\lambda>0$,
\begin{equation}\label{prelim equation 7}
\rho_{\pp}(f/\lambda)=
\int_{X\setminus \Omega_{\infty,\pp}}
\left(\frac{|f(x)|}{\lambda}\right)^{p(x)}\,d\mu(x)+
\lambda^{-1}\|f\|_{L^{\infty}(\Omega_{\infty,\pp})}<\infty.
\end{equation}
$\Lp(X)$ is a Banach function space when equipped with the Luxemburg
norm
\begin{equation}\label{prelim equation 8}
\|f\|_{\pp}=\inf\{\lambda>0:\rho_{\pp}(f/\lambda)\leq1\}.
\end{equation}

For the fractional maximal operator to be bounded, we need to impose
some restrictions on the exponent function $\pp$.  To state them, we
will need a simple measure of the oscillation of $\pp$.  Given a set
$E\subset X$, we define
\[ p_+(E) = \esssup_{x\in E} p(x), \qquad p_-(E) = \essinf_{x\in E}
p(x).  \]
For brevity we write $p_+=p_+(X)$ and $p_-=p_-(X)$.  We will also need
to control the continuity of $\pp$ locally and at infinity.   Our
hypothesis is the same log-H\"older continuity condition that has played an
important role in the Euclidean case:  see~\cite{MR3026953} for
details and further references.  

\begin{definition}\label{prelim definition 5}
Given a function $r(\cdot):X\rightarrow[0,\infty)$, we say that $r(\cdot)$ is 
locally log-H\"{o}lder continuous and write $r(\cdot)\in LH_0$ if there exists a constant 
$C_{0}$ such that
\[
|r(x)-r(y)|\leq
\frac{-C_{0}}{\log d(x,y)},
\]
where $x,\,y\in X$ and $d(x,y)<1/2$. The constant $C_0$ is called the
$LH_0$ constant of $r(\cdot)$.
\end{definition}

\begin{definition}\label{prelim definition 6}
Given a function $r(\cdot):X\rightarrow[0,\infty)$, we say that $r(\cdot)$ is 
log-H\"{o}lder continuous at infinity with respect to a base point $x_0\in X,$ and write $r(\cdot)\in LH_\infty$, if there exist constants 
$C_\infty$, $r_\infty$ such that
\[
|r(x)-r_\infty|\leq
\frac{C_\infty}{\log(e+d(x,x_0))},
\]
for every $x\in X$.  The constant $C_\infty$ is called the $LH_\infty$
constant of $r(\cdot)$.
\end{definition}

When  $r(\cdot)\in LH_0\bigcap LH_\infty$ we say it is globally log-H\"{o}lder 
continuous and we define $LH= LH_0\bigcap LH_\infty$.

\begin{remark}
Since we wish to allow for the possibility of unbounded exponents $\pp$, we will apply 
the $LH_0$, $LH_\infty$ and $LH$ conditions to the function $1/\pp$ instead of applying them
to $\pp$.
\end{remark}

The definition of the $LH_\infty$ condition assumes the existence of a base point 
$x_0$ which is taken to be the origin in the Euclidean case. On a
general space of homogeneous type, there may not be such a
distinguished point; however, the choice of the base point is
immaterial as the next lemma shows.  We refer the reader
to~\cite{MR3322604} for a proof.  

\begin{lemma}\label{prelim lemma 4}
  Let $r(\cdot)\in LH_\infty$ with respect to the base point
  $x_0\in X$.   Given any $y_0\in X$, we have that $r(\cdot)\in LH_\infty$
 with respect to $y_0$ with a possibly different $LH_\infty$ constant.  
\end{lemma}

In the calculations to follow, we will need to estimate certain integrals that appear as error
terms. The following result was proved in 
\cite{MR3322604};  we include the short proof for completeness.

\begin{lemma}\label{prelim lemma 5}
Let $(X,d,\mu)$ be a space of homogeneous type.
If $N>\log_{2} C_{\mu}$, then for any $x_0\in X$,
\[
\underset{X}{\int}\frac{1}{(e+d(x,x_0))^{N}}\,d\mu<+\infty.
\]
\end{lemma}

\begin{proof}
Fix $x_0 \in X$,  define $B_{n}=B(x_0,2^{n})$ for $n\geq 0$
and let $B_{-1}=\emptyset$. 
Then $X=\underset{n\geq 0}{\bigcup}(B_{n}\setminus B_{n-1})$ and so we
have that 
\begin{align*}
\int_X\frac{1}{(e+d(x,x_0))^{N}}\,d\mu
&= \underset{n\geq 0}{\sum}
\int_{B_{n}\setminus B_{n-1}}
\frac{1}{(e+d(x,x_0))^{N}}\,d\mu\\
&\leq
\frac{1}{e^{N}}\mu(B_0)+
\underset{n\geq 1}{\sum}
\frac{1}{(e+2^{n-1})^{N}}\mu(B_{n}\setminus B_{n-1})\\
&\leq
\frac{1}{e^{N}}\mu(B_0)+
\underset{n\geq 1}{\sum}
C_{\mu}^{n}\frac{\mu(B_0)}{2^{(n-1)N}}\\
&\leq
\mu(B_0)\bigg(1+C_{\mu}\underset{n\geq 0}{\sum}
\left(\frac{C_{\mu}}{2^{N}}\right)^{n}\bigg).
\end{align*}
Since $\log_{2} C_{\mu} < N$ the final summation converges.
\end{proof}

\begin{remark}\label{prelim remark 1}
Below we will use Lemma \ref{prelim lemma 5} as follows:
fix $0<\Gamma<1$.  Then we may write
\[
\int_{X}\Gamma^{C_{\infty}^{-1}\log(e+d(x,x_0))}\,d\mu(x)=
\int_{X}
\frac{1}{(e+d(x,x_0))^{C_{\infty}^{-1}\log(1/\Gamma)}}\,d\mu(x).
\]
If $C_{\infty}^{-1}\log(1/\Gamma)>\log_{2} C_\mu$, then by
Lemma~\ref{prelim lemma 5} the integral on the right 
converges. Therefore, by the dominated convergence theorem,
\[
\underset{\Gamma\rightarrow 0}{\lim}
\int_{X}\Gamma^{C_{\infty}^{-1}\log(e+d(x,x_0))}\,d\mu(x)=0.
\]
In particular, we can always fix a constant $0<\Gamma<1$ so that the
integral of $\Gamma^{C_{\infty}^{-1}\log(e+d(\cdot,x_0))}$ over $X$ is
smaller than any given positive number.
\end{remark}

\medskip

We will use the $LH_\infty$ condition to replace variable exponents
with constant ones.  The following lemma was proved in~\cite{MR2493649}
in the Euclidean case. We include the short proof for completeness. 

\begin{lemma}\label{prelim lemma 6}
Let $(X,d,\mu)$ be a space of homogeneous type.
Given $r(\cdot)\in LH_\infty$, suppose $r_{\infty}>0$. Fix $x_0\in X$
and define $R(x)=(e+d(x,x_0))^{-N}$, where $N r_\infty>\log_{2} C_{\mu}$. 
Then given any measurable set $E\subset X$ and any measurable function
$F$ such that $0\leq F(y)\leq 1$ for a.e. $y\in E$, 
\begin{gather}
\int_EF(y)^{r(y)}\,d\mu(y)\leq
e^{N C_\infty}\int_{E}F(y)^{r_\infty}\,d\mu(y)+
e^{N C_\infty}\int_ER(y)^{r_\infty}\,d\mu(y), \label{eqn:LHinfty1}\\
\int_EF(y)^{r_\infty}\,d\mu(y)\leq
e^{N C_\infty}\int_{E}F(y)^{r(y)}\,d\mu(y)+
\int_ER(y)^{r_\infty}\,d\mu(y). \label{eqn:LHinfty2}
\end{gather}
\end{lemma}

\begin{proof}
We shall prove inequality \eqref{eqn:LHinfty2}; the proof of \eqref{eqn:LHinfty1}
is similar.  Write
\[
\int_EF(y)^{r_\infty}\,d\mu(y)=
\int_{E_{1}}F(y)^{r_\infty}\,d\mu(y)+
\int_{E_{2}}F(y)^{r_\infty}\,d\mu(y),
\]
where 
\[
E_1=\{y\in E:F(y)\leq R(y)\},\qquad 
E_2=\{y\in E:F(y)>R(y)\}.
\]
On the set $E_1$, $F(y)^{r_\infty}\leq R(y)^{r_\infty}$;  hence,
\[
\int_{E_{1}}F(y)^{r_\infty}\,d\mu(y)\leq
\int_{E_{1}}R(y)^{r_\infty}\,d\mu(y).
\]

To estimate the integral over the set $E_2$, note that 
since $0\leq F(y)\leq 1$ for a.e. $y\in E_2$, 
$F(y)^{r_{\infty}-r(y)}\leq R(y)^{-|r(y)-r_{\infty}|}$.   Thus
\[
\int_{E_{2}}F(y)^{r_\infty}\,d\mu(y)=
\int_{E_{2}}F(y)^{r_\infty} F(y)^{r_{\infty}-r(y)}\,d\mu(y)
\leq
\int_{E_{2}}F(y)^{r(y)} R(y)^{-|r(y)-r_{\infty}|}\,d\mu(y).
\]
Since $r(\cdot)\in LH_\infty$, 
\[
R(y)^{-|r(y)-r_{\infty}|}=e^{N |r(y)-r_{\infty}|\log(e+d(y,x_0))}\leq
e^{N C_{\infty}}.
\]
If we combine all these estimates, we get \eqref{eqn:LHinfty2}.
\end{proof}

\begin{remark}
  If $\mu(E)<\infty$, then the integral of $R(y)$ is always finite.
  If $\mu(E)=\infty$, then by Lemma \ref{prelim lemma 5} the
  assumption that $N r_\infty>\log_2C_\mu$ ensures
\[
\int_{E}R(y)^{r_\infty}\,d\mu(y)\leq
\int_{X}\frac{1}{(e+d(y,x_0))^{N r_\infty}}\,d\mu(y)<+\infty.
\]
\end{remark}

\medskip

We will use the $LH$ condition to get estimates on the measure of
cubes.  The following result is referred to as a Diening type
estimate: see \cite{MR3322604} for a proof and the history of this
important condition. In the Euclidean case this estimate is equivalent
to the $LH_0$ condition; in the more general setting of spaces of
homogeneous type we seem to require a stronger hypothesis.

\begin{lemma}\label{prelim lemma 7}
  Given a space of homogeneous type $(X,d,\mu)$, let
  $\pp:X\rightarrow\left[1,\infty\right]$ be such
  that $1/\pp\in LH$. Then there is a positive constant $C$ such that
  for any ball $B$
\begin{enumerate}
\item  $\mu(B)^{\frac{1}{p_{+}(B)}-\frac{1}{p_{-}(B)}}\leq C$;\\

\item for all $x\in B$,  $\mu(B)^{\frac{1}{p(x)}-\frac{1}{p_{-}(B)}}\leq C$ and 
                       $\mu(B)^{\frac{1}{p_{+}(B)}-\frac{1}{p(x)}}\leq C$.
\end{enumerate}
\end{lemma}

We will actually need the analog of the Diening estimate for dyadic cubes.  

\begin{cor}\label{prelim corollary 1}
  Given a space of homogeneous type $(X,d,\mu)$, let
  $\pp:X\rightarrow\left[1,\infty\right]$ be such that $1/\pp\in LH$
  and let $\D$ be a dyadic grid on $X$.  Then inequalities $(1)$ and
  $(2)$ of Lemma~\textup{\ref{prelim lemma 7}} hold with $B$ replaced by any dyadic cube
  $Q\in\D$.
If $\mu(X)<+\infty$, then the same conclusion holds if $1/\pp\in LH_0$.  
\end{cor}

\begin{proof}
  Let $Q$ be a dyadic cube. If $\mu(Q)\geq1$, then the result follows
  trivially. Assume $\mu(Q)<1$. From Theorem \ref{prelim theorem 1} it
  follows that there are balls $B_1$ and $B_2$ with the same centers
  and comparable radii such that $B_1\subseteq Q\subseteq B_2$. By
  the lower mass bound \eqref{prelim equation 1}, there exists
 a constant $K\geq1$ such that
  $1/K\leq\mu(B_1)/\mu(B_2)\leq\mu(Q)/\mu(B_2)$.  Since
  $Q\subseteq B_2$ it follows that
  $1/p_{-}(Q)-1/p_{+}(Q)\leq1/p_{-}(B_2)-1/p_{+}(B_2)$. Since $K\geq1$
  and $\mu(Q)<1$, by Lemma \ref{prelim lemma 7} we have that
\[
\left(\frac{1}{\mu(Q)}\right)^{\frac{1}{p_{-}(Q)}-\frac{1}{p_{+}(Q)}}\leq
K^{\frac{1}{p_-}-\frac{1}{p_+}}
\left(\frac{1}{\mu(B_2)}\right)^{\frac{1}{p_{-}(B_2)}-\frac{1}{p_{+}(B_2)}}\leq
KC.
\]
This proves $(1)$.  The proof of $(2)$ is essentially the same.

\medskip

If $\mu(X)<+\infty$, then by Lemma~\ref{prelim finite measure},
$\diam(X) < +\infty$, Therefore, there exists $x_0\in X$ and $R>0$ such that $X\subset
B(x_0,R)$.  Let $p_\infty=p_-$.  Then for any $x\in X$,
\[ \left|\frac{1}{p(x)}-\frac{1}{p_\infty}\right|
\leq \frac{2}{p_-} \frac{\log(e+d(x,x_0))}{\log(e+d(x,x_0))}
\leq \frac{2\log(e+R)}{\log(e+d(x,x_0))}. \]
Hence $1/\pp$ satisfies the $LH_\infty$ condition with $C_\infty
=2\log(e+R)$.
\end{proof}

Finally, we will need a version of the monotone convergence theorem for
variable Lebesgue spaces.  For a proof in the Euclidean case,
see~\cite{MR3026953}; the same proof holds without change on spaces of homogeneous type.

\begin{lemma}\label{prelim lemma 8}
  Given a space of homogeneous type $(X,d,\mu)$ and a non-negative
  function $f\in \Lp(X)$, suppose that the sequence $\{f_N\}$ of
  non-negative functions increases pointwise to $f$ almost
  everywhere. Then $\|f_N\|_{\pp}$ increases to $\|f\|_{\pp}$.
\end{lemma}

\section{Proof of the boundedness of $M_\eta$}
\label{main}

In this section we prove Theorem~\ref{main theorem 1}.
We will first assume that $\mu(X)=\infty$.  The proof when
$\mu(X)<\infty$ is similar but shorter, and we will prove it at the
end of the section.

\medskip

 We begin the proof with some reductions.  First, we may assume that
 $p_-<\infty$.  If $p_-=\infty$, then $\pp =\infty$ a.e. and hence $\eta=0$.
In this case Theorem~\ref{main theorem 1} reduces to the elementary
fact that the maximal operator is bounded on $L^\infty(X)$.  

Second, by Proposition \ref{prelim proposition 1} it will suffice to
prove inequality \eqref{main equation 1} with $M_{\eta}$ replaced by
$M^{\D}_{\eta}$ for an arbitrary dyadic grid $\D$ on $X$.

Third, by the definitions of variable Lebesgue norms and dyadic
fractional maximal operators, we may assume without loss of generality
that $f$ is non-negative.  Moreover, we may assume that $f$ is a
bounded function with bounded support. If inequality \eqref{main equation 1}
holds for such functions, then given an arbitrary non-negative function 
$f\in \Lp(X)$ and a base point $x_0$, then it holds for the functions 
$f_n(x)=\min\{f(x),n\} \chi_{B(x_0,n)}(x)$ and by Lemmas~\ref{prelim lemma 9} 
and~\ref{prelim lemma 8} we have 
\[ \|M^\D_\eta f\|_\qq = \lim_{n\rightarrow \infty}
\|M^\D_\eta f_n\|_\qq
\leq C\lim_{n\rightarrow \infty}
\| f_n\|_\pp = \|f\|_\pp. \]

Finally, by the homogeneity of the norm we may assume that
$\|f\|_\pp=1$. 

\medskip

Fix such a function $f$ and write $f=f_1+f_2$, where $f_1=f \chi_{\{f>1\}}$
and $f_2=f \chi_{\{f\leq 1\}}$.  Then 
\[ \|M_\eta^\D f \|_\qq \leq \|M_\eta^\D f_1\|_\qq
+ \|M_\eta^\D f_2\|_\qq, \]
and for $i=1,\,2$, $\|f_i\|_\pp \leq \|f\|_\pp=1$.  Therefore, it will
suffice 
to find positive constants $\lambda_i$ independent of $\D$ such that
$\|M^{\D}_{\eta} f_i\|_{\qq}\leq\lambda_i$ for $i=1,2$.  By the
definition of variable Lebesgue space norms, this is equivalent to
finding $\lambda_i$ such that 
$\rho_{\qq}(\lambda_{i}^{-1} M^{\D}_{\eta} f_i)\leq1$ for $i=1,2$.

We will also use the fact that since 
$\|f\|_{\pp}=1$  for $i=1,\,2$, we have that $\rho_{\pp}(f_i) \leq
\rho_{\pp}(f)\leq\|f\|_{\pp}=1$.   This follows from the definition of
the norm:  see \cite{MR3026953,DHHR} for details.

\subsection*{The estimate for $f_1$}\label{main subsection 1}

This part of the argument closely follows the proof given
in~\cite{MR2493649}.  Since $\rho_{\pp}(f_1)\leq1$, by the definition
of the modular,
$\|f_1\|_{L^{\infty}(\Omega_{\infty}^{\pp})}\leq1$. Hence, $f_1\leq1$
almost everywhere on $\Omega_{\infty}^{\pp}$.  But by definition
$f_1(x)>1$ or $f_1(x)=0$, so $f_1=0$ a.e.~on
$\Omega_{\infty}^{\pp}$. In other words, up to a set of measure zero
$\supp(f_1)\subset X\setminus \Omega_{\infty}^{\pp}$.

We want to show that there exists a constant $\lambda_1>1$
such that $\rho_{\qq}(\lambda_{1}^{-1}M^{\D}_{\eta}
f_1)\leq1$.   For the argument below it is convenient
to write $\lambda_{1}^{-1}=\alpha_{1}
\beta_{1} \gamma_{1}$, where we will assume $0<\alpha_1,\,\beta_1,\,\gamma_1<1$.
In fact, we will show that these constants can be chosen so that
\begin{equation}\label{main subsection 1 equation 1}
\int_{X\setminus \Omega^{\qq}_{\infty}}(\alpha_1 \beta_1 \gamma_1 M^{\D}_{\eta} f_1(x))^{q(x)}\,d\mu(x)
\leq\frac{1}{2} 
\end{equation}
and
\begin{equation}\label{main subsection 1 equation 2}
\alpha_1 \beta_1 \gamma_1 \|M^{\D}_{\eta} f_1\|_{L^{\infty}(\Omega^{\qq}_{\infty})}\leq\frac{1}{2}.
\end{equation}

\medskip

We first prove inequality (\ref{main subsection 1 equation 1}).  If
$\eta=0$, then for every $x\in X$, 
$M^{\D} f_1(x)\leq\|f_1\|_{\infty}<+\infty$. If $0<\eta<1$ and $1<r<1/\eta$, then
$f_1\in L^{r}(X)\bigcap L^{\infty}(X)$ because $f_1$ is bounded and
has bounded support.  Hence, by Remark
\ref{prelim remark 2} and inequality (\ref{prelim equation 14}), 
$M^{\D}_{\eta} f_1(x)<+\infty$.  

Let $C$ be the constant in Lemma \ref{prelim lemma 3}. For each integer $k$ define the set 
\[ \Omega_k=
\{x\in X:M^{\D}_{\eta} f_1(x)>C^{k}\}. \]
Since $M^{\D}_{\eta} f_1$ is finite and positive everywhere, 
\[
X=\underset{k}{\bigcup}(\Omega_{k}\setminus \Omega_{k+1}).
\]
Since $f_1$ satisfies the hypotheses of Lemma \ref{prelim lemma 3}
we can write 
\[ 
\Omega_k=
\{x\in X:M^{\D}_{\eta} f_1(x)>C^{k}\}=
\underset{j}{\bigcup}Q^{k}_{j},
\]
where the dyadic cubes $Q^{k}_{j}$ satisfy
\[ 
\mu(Q^{k}_{j})^{\eta}\avgint_{Q^{k}_{j}}f_1\,d\mu>
C^{k}.
\]
It follows by the properties of dyadic cubes that if we
define $E_j^k=Q_j^k\setminus \Omega_{k+1}$, then the sets $E_j^k$ are
pairwise disjoint.   (See~\cite{MR3302574} where this fact is given using
somewhat different terminology.)

We can now estimate as follows:  

\begin{align}
& \notag\int_{X\setminus \Omega_{\infty}^\qq}
(\alpha_1 \beta_1 \gamma_1 M^{\D}_{\eta} f_1(x))^{q(x)}\,d\mu(x) \\
& \notag \qquad \qquad =
\underset{k}{\sum}\int_{(\Omega_{k}\setminus \Omega_{k+1})\setminus \Omega_{\infty}^{\qq}}
(\alpha_1 \beta_1 \gamma_1 M^{\D}_{\eta} f_1(x))^{q(x)}\,d\mu(x)\\
& \notag \qquad \qquad \leq
\underset{k}{\sum}\int_{(\Omega_{k}\setminus \Omega_{k+1})\setminus \Omega_{\infty}^{\qq}}(\alpha_1 \beta_1 \gamma_1 C^{k+1})^{q(x)}\,d\mu(x).
\intertext{If we choose $0<\alpha_1\leq 1/C$, then}
&  \label{eqn:first} \qquad \qquad \leq
\underset{k,j}{\sum}\int_{E^{k}_{j}\setminus \Omega_{\infty}^{\qq}}
\left(\beta_1 \gamma_1\mu(Q^{k}_{j})^{\eta}\avgint_{Q^{k}_{j}}f_1(y)\,d\mu(y)\right)^{q(x)}\,d\mu(x).
\end{align}

If $k$ and $j$ are such that
$\mu(E^{k}_{j}\setminus \Omega_{\infty}^{\qq})=0$, then this term in
the sum is zero.  Hence, we may disregard those terms and assume that
$\mu(E^{k}_{j}\setminus \Omega_{\infty}^{\qq})>0$.  
Since
$E^{k}_{j}\subseteq Q^{k}_{j}$,  we have that the exponents 
$p_{jk}=p_-(Q_j^k)$ and
$q_{jk}=q_-(Q_j^k)$ are both finite and satisfy 
$1<p_{jk}<(1/\eta)$ and $(1/p_{jk})-(1/q_{jk})=\eta$.

Therefore, by Lemma \ref{prelim lemma 10} we have 
\[
\mu(Q^{k}_{j})^{\eta}\avgint_{Q^{k}_{j}}f_1\,d\mu\leq
\left(\int_{Q^{k}_{j}}f_{1}^{p_{jk}}\,d\mu\right)^{\frac{1}{p_{jk}}-\frac{1}{q_{jk}}}
\left(\avgint_{Q^{k}_{j}}f_1\,d\mu\right)^{\frac{p_{jk}}{q_{jk}}}.
\]
Since $\supp(f_1)\subset X\setminus \Omega_{\infty}^{\pp}$ (up to a
set of measure zero), and since $f_1$ is either $0$ or greater than
$1$, it follows that the first integral on the righthand side is dominated by
$\rho_{\pp}(f_1)$
and so it is less than or equal to $1$.
Therefore, by H\"{o}lder's inequality with respect to the exponent $p_{jk}/p_-$
we get
\[ 
\mu(Q^{k}_{j})^{\eta}\avgint_{Q^{k}_{j}}f_1(y)\,d\mu(y)
\leq
\mu(Q^{k}_{j})^{-\frac{p_-}{q_{jk}}}
\left(\int_{Q^{k}_{j}}f_1(y)^{\frac{p_{jk}}{p_{-}}}\,d\mu(y)\right)^{\frac{p_-}{q_{jk}}}.
\]
If we  substitute this  inequality into 
inequality \eqref{eqn:first} and rearrange terms, we get
\begin{multline}\label{main subsection 1 equation 11}
\int_{X\setminus \Omega_{\infty}^\qq}
(\alpha_1 \beta_1 \gamma_1 M^{\D}_{\eta} f_1(x))^{q(x)}\,d\mu(x)\\
\leq
\underset{k,j}{\sum}\int_{E^{k}_{j}\setminus \Omega_{\infty}^{\qq}}
(\beta_1 \mu(Q^{k}_{j})^{-\frac{p_-}{q_{jk}}})^{q(x)}
\left(\gamma_1 \left(\int_{Q^{k}_{j}}f_1(y)^{\frac{p_{jk}}{p_{-}}}\,d\mu(y)\right)^{\frac{p_-}{q_{jk}}}\right)^{q(x)}\,d\mu(x).
\end{multline}

Since $1/\qq\in LH$, by Corollary~\ref{prelim corollary 1}
there exists a constant $C$ such that for any $x\in Q^{k}_{j}$, 
\[
\beta_1 \mu(Q^{k}_{j})^{-\frac{p_-}{q_{jk}}}\leq
\beta_1 C^{p_{-}} \mu(Q^{k}_{j})^{-\frac{p_-}{q(x)}}.
\]
If we choose $0<\beta_1\leq 1/C^{p_{-}}$, then
\[ 
\beta_1 \mu(Q^{k}_{j})^{-\frac{p_-}{q_{jk}}}\leq
\mu(Q^{k}_{j})^{-\frac{p_-}{q(x)}}
\]
for any $x\in E_j^k \subset Q^{k}_{j}$.  If we substitute this estimate
into \eqref{main subsection 1 equation 11}, we get
\begin{align*}
&\int_{X\setminus \Omega_{\infty}^\qq}
(\alpha_1 \beta_1 \gamma_1 M^{\D}_{\eta} f_1(x))^{q(x)}\,d\mu(x)\\
& \qquad \qquad \leq
\underset{k,j}{\sum}\int_{E^{k}_{j}\setminus \Omega_{\infty}^{\qq}}
\mu(Q^{k}_{j})^{-p_{-}}
\left(\gamma_1 \left(\int_{Q^{k}_{j}}f_1(y)^{\frac{p(y)}{p_{-}}}\,d\mu(y)\right)^{\frac{p_-}{q_{jk}}}\right)^{q(x)}\,d\mu(x).
\intertext{ Since $q_{jk}\geq
p_{jk}\geq p_{-}$ and $\gamma_{1}\leq\gamma_1^{\frac{p_-}{q_{jk}}}$,
  we have that}
&\qquad \qquad \leq
\underset{k,j}{\sum}\int_{E^{k}_{j}\setminus \Omega_{\infty}^{\qq}}
\mu(Q^{k}_{j})^{-p_{-}}
\left(\gamma_1 \int_{Q^{k}_{j}}f_1(y)^{\frac{p(y)}{p_{-}}}\,d\mu(y)\right)^{q(x) \frac{p_-}{q_{jk}}}\,d\mu(x).
\intertext{ Since $f_1\geq1$ or $f_1=0$, $\supp(f_1)\subset X\setminus
  \Omega_{\infty}^{\pp}$ (up to a set of measure zero),
  $\rho_{\pp}(f_1)\leq1$ and $p_{-}>1$.  Hence, the quantity inside the parentheses is less 
than or equal to $1$. Since $q(x)\geq q_{jk}$, we get}
&\qquad \qquad \leq
\underset{k,j}{\sum}\int_{E^{k}_{j}\setminus \Omega_{\infty}^{\qq}}
\mu(Q^{k}_{j})^{-p_{-}}
\left(\gamma_1 \int_{Q^{k}_{j}}f_1(y)^{\frac{p(y)}{p_{-}}}\,d\mu(y)\right)^{p_{-}}\,d\mu(x)\\
&\qquad \qquad =
\underset{k,j}{\sum}\int_{E^{k}_{j}\setminus \Omega_{\infty}^{\qq}}
\left(\gamma_1 \avgint_{Q^{k}_{j}}f_1(y)^{\frac{p(y)}{p_{-}}}\,d\mu(y)\right)^{p_{-}}\,d\mu(x)\\
&\qquad \qquad \leq
\int_{X\setminus \Omega_{\infty}^\qq} \gamma_{1}^{p_{-}}
M^{\D}(f_{1}(\cdot)^{\frac{p(\cdot)}{p_{-}}})(x)^{p_{-}}\,d\mu(x).
\intertext{ Since $p_{-}>1$,  $M^{\D}$ is bounded on $L^{p_{-}}$ and
  there exists a constant $C=C(p_-,X)$ such that }
&\qquad \qquad \leq
\gamma_{1}^{p_{-}}C\int_{X\setminus \Omega_\infty^\pp} 
f_{1}(x)^{p(x)}\,d\mu(x)\\
&\qquad \qquad \leq
\gamma_{1}^{p_{-}}C \rho_\pp(f_1) \\
&\qquad \qquad \leq
\gamma_{1}^{p_{-}}C.
\end{align*}
Thus, inequality~\eqref{main subsection 1 equation 1} holds if we 
choose $0 < \gamma_1 \leq 2 C^{-1/p_-}$.

\medskip

We now prove inequality (\ref{main subsection 1 equation 2}).  We will
show that 
\[
\frac{1}{2} \|M^{\D}_{\eta} f_1\|_{L^{\infty}(\Omega_{\infty}^{\qq})}\leq
C;
\]
this gives us inequality \eqref{main subsection 1 equation 2} 
if we choose $\alpha_1,\,\beta_1,\,\gamma_1$ such that 
$0<\alpha_1\beta_1\gamma_1<1/(4 \max\{1,C\})$.

Without loss of generality we may assume that $\|M^{\D}_{\eta} f_1\|_{L^{\infty}(\Omega_{\infty}^{\qq})}>0$. 
By definition there 
exists $x\in \Omega_{\infty}^{\qq}$ and a dyadic cube $Q$ such that $x\in Q$ and
\[
\frac{1}{2} \|M^{\D}_{\eta} f_1\|_{L^{\infty}(\Omega_{\infty}^{\qq})}<
\mu(Q)^{\eta}\avgint_{Q}f_1\,d\mu.
\]
Since $f_1=0$ almost everywhere on $\Omega_{\infty}^{\pp}$ it follows that 
$\mu(Q\setminus \Omega_{\infty}^{\pp})>0$. This implies that
$1<p_-(Q)<+\infty$. We consider three cases:
\begin{enumerate}
\item $p_-(Q)=1/\eta$ and $0<\eta<1$,
\item $1<p_-(Q)<1/\eta$ and  $0<\eta<1$,  
\item $1<p_-(Q)<+\infty$ and  $\eta=0$.
\end{enumerate}
If $(1)$ holds, an application of H\"older's inequality with respect to the 
exponent $1/\eta$, together with the facts that $f_1$ is either $0$ or greater than $1$,  
$\supp(f_1)\subset X\setminus \Omega_{\infty}^{\pp}$ (up to a set of
measure zero) and $\rho_\pp(f_1)\leq 1$, gives 
\[
\frac{1}{2}\|M^{\D}_{\eta} f_1\|_{L^{\infty}(\Omega_{\infty}^{\qq})}
\leq \mu(Q)^{\frac{1}{p_-(Q)}}\left(\avgint_Q f_1^{p_-(Q)}\,d\mu\right)
^{\frac{1}{p_-(Q)}}
\leq \left(\avgint_Q f_1^{p(x)}\,d\mu\right)
^{\frac{1}{p_-(Q)}}
\leq 1.
\]

If $(2)$ holds, then by Lemma \ref{prelim lemma 10} and arguing as in the first case, we have that 
\begin{multline}
\frac{1}{2}\|M^{\D}_{\eta} f_1\|_{L^{\infty}(\Omega_{\infty}^{\qq})}
< \mu(Q)^{\eta}\avgint_{Q}f_1\,d\mu \\
 \leq\left(\int_Qf_{1}^{p_-(Q)}\,d\mu\right)^{\frac{1}{p_-(Q)}-\frac{1}{q_-(Q)}}
\left(\avgint_{Q}f_1\,d\mu\right)^{\frac{p_-(Q)}{q_-(Q)}} \\
 \leq \left(\avgint_{Q}f_1^{p_-(Q)}\,d\mu\right)^{\frac{1}{q_-(Q)}} 
 \leq \left(\frac{1}{\mu(Q)}\right)^{\frac{1}{q_-(Q)}}.  \label{eqn:infty-bound}
\end{multline}

Finally, when $(3)$ holds we get the same estimate as in Case $(2)$;
to prove it simply 
replace $M^{\D}_{\eta}$ with $M^{\D}$ and $q_-(Q)$ with $p_-(Q)$ and
argue as before.  

\medskip

To complete the proof we must show that the last term in~\eqref{eqn:infty-bound} is uniformly
bounded. If $\mu(Q)\geq1$, this is immediate.  Now assume
$\mu(Q)<1$ and let $m$ be an arbitrary positive integer. Since $1/\qq \in LH$ and $x\in \Omega_{\infty}^{\qq}$ it follows
that the set
\[
\left\{y\in X:\frac{1}{q(y)}<\frac{1}{m}\right\}
\]
is non-empty and open. Thus there exists a $\tau>0$ such that the ball $B(x,\tau)$ is 
contained in it.

Let the constants $C$ and $\delta$ be the same as in Theorem \ref{prelim theorem 1}. 
Choose an integer $k$ such that $2 A_0 C \delta^{k}<\tau$. Since the dyadic cubes in the $k$-th 
generation cover $X$ it follows that there exists a dyadic cube $\widetilde{Q}$ in the $k$-th generation
which contains $x$.  By choosing $k$ large enough we may assume that
$\widetilde{Q}\subset Q$.  Then by the quasi-triangle inequality, 

\[
x\in \widetilde{Q}\subseteq B(x_c(\widetilde{Q}),C \delta^{k})
\subseteq B(x,2 A_0 C \delta^{k})\subseteq
B(x,\tau);
\] 
in other words,
\[
\widetilde{Q}\subseteq
\left\{y\in X:\frac{1}{q(y)}<\frac{1}{m}\right\}.
\]
Hence, 
$q_+(Q)\geq q_+(\widetilde{Q})\geq m$.
Since the positive integer $m$ was arbitrary, it follows that $q_+(Q)=+\infty$.

Now let $n$ be an arbitrary positive integer such that
$q_-(Q)<n$. Since $q_+(Q)=+\infty$, 
there exists a point $x_n\in Q$ such that 
$q_-(Q)<n<q(x_n)<q_+(Q)=+\infty$. Because $1/\qq\in LH$ , by 
Corollary \ref{prelim corollary 1}, 
\[
\mu(Q)^{\frac{1}{q(x_n)}-\frac{1}{q_-(Q)}}\leq C.
\] 
Therefore, since $\mu(Q)<1$,  we have that 
\[
\frac{1}{2} \|M^{\D}_{\eta} f_1\|_{L^{\infty}(\Omega_{\infty}^{\qq})}\leq
C \left(\frac{1}{\mu(Q)}\right)^{\frac{1}{q(x_n)}} 
\leq C \left(\frac{1}{\mu(Q)}\right)^{\frac{1}{n}}.
\]
If we take the limit as $n\rightarrow\infty$, we get the desired
bound.  This completes the proof of estimate~\eqref{main subsection 1 equation 2}.

\subsection*{The estimate for $f_2$}\label{main subsection 2}

Recall that the function $f_2$ has bounded support, $0\leq f_2\leq1$ and $\rho_{\pp}(f_2)\leq1$.
We want to prove that there exists $\lambda_2>1$ such that 
$\rho_{\qq}(\lambda_{2}^{-1}M^{\D}_{\eta} f_2)\leq1$.   Let
$\alpha_2=\lambda_2^{-1}$.  We will
show that there exists $0<\alpha_2<1$ such that 

\begin{equation}\label{main subsection 2 equation 1}
\int_{X\setminus \Omega^{\qq}_{\infty}}(\alpha_2 M^{\D}_{\eta} f_2(x))^{q(x)}\,d\mu(x)
\leq\frac{1}{2}
\end{equation}
and
\begin{equation}\label{main subsection 2 equation 2}
\alpha_2 \|M^{\D}_{\eta} f_2\|_{L^{\infty}(\Omega^{\qq}_{\infty})}\leq\frac{1}{2}.
\end{equation}

To prove both of these inequalities we first give
a pointwise estimate
for $M^{\D}_{\eta} f_2$. Let $x\in X$ and let $Q$ be any dyadic cube
containing $x$. If 
$\mu(Q\bigcap \Omega_{\infty}^{\pp})>0$, then $\eta=0$ because $p_+(Q)=+\infty$ and  
$p_+(Q)\leq 1/\eta$. Hence,
\[ 
\mu(Q)^{\eta} \avgint_{Q}f_2(y)\,d\mu(y)=
\avgint_{Q}f_2(y)\,d\mu(y)\leq1.
\]

If $\mu(Q\bigcap \Omega_{\infty}^{\pp})=0$ and $\eta=0$,  it is
immediate that the same
inequality holds.   Now suppose that
$\mu(Q\bigcap \Omega_{\infty}^{\pp})=0$ and $0<\eta<1$. In this case H\"{o}lder's 
inequality with respect to the exponent $1/\eta$ gives
\begin{multline*}
\left(\mu(Q)^{\eta} \avgint_{Q}f_2(y)\,d\mu(y)\right)^{1/\eta}
\leq \int_{Q}f_{2}(y)^{1/\eta}\,d\mu(y)\\
=\int_{Q\setminus \Omega_{\infty}^{\pp}}f_{2}(y)^{1/\eta}\,d\mu(y)
\leq \int_{Q\setminus \Omega_{\infty}^{\pp}}f_{2}(y)^{p(y)}\,d\mu(y)
\leq 1.
\end{multline*}
The last inequalities hold since $p(y)\leq p_+\leq 1/\eta$, $f_2\leq1$
and $\rho_{\pp}(f_2)\leq1$. Therefore, in every case the fractional
average of $f_2$ is at most $1$ and so $M^{\D}_{\eta} f_2(x)\leq1$.

\medskip

Given this estimate, we can immediately prove (\ref{main subsection 2
  equation 2}):   choose $0<\alpha_2\leq 1/2$; then 
\[ \alpha_2 \|M^{\D}_{\eta}
f_2\|_{L^{\infty}(\Omega_{\infty}^{\qq})}\leq\frac{1}{2}. \]

\medskip

To prove \eqref{main subsection 2  equation 1} we will consider two
cases: $q_\infty<+\infty$ and $q_\infty=+\infty$.
First suppose that $q_\infty<+\infty$. 
Since $1/\qq\in LH$, by Lemma~\ref{prelim lemma 6} we get
\begin{multline}\label{main subsection 2 equation 5}
\int_{X\setminus \Omega_{\infty}^{\qq}}(\alpha_2 M^{\D}_{\eta} f_2(x))^{q(x)}\,d\mu(x)=
\int_{X\setminus \Omega_{\infty}^{\qq}}((\alpha_2 M^{\D}_{\eta} f_2(x))^{q(x)q_\infty})^{\frac{1}{q_\infty}}\,d\mu(x)\\
\leq
e^{N C_\infty}\int_{X\setminus \Omega_{\infty}^{\qq}}(\alpha_2 M^{\D}_{\eta} f_2(x))^{q_\infty}\,d\mu(x)+
\int_{X\setminus \Omega_{\infty}^{\qq}}R(x)^{\frac{1}{q_\infty}}\,d\mu(x),
\end{multline}
where $R(x)=(e+d(x,x_0))^{-N}$.   By Lemma \ref{prelim lemma 5} and the dominated convergence 
theorem we may choose $N$ large enough that 
\begin{equation}\label{main subsection 2 equation 6}
\int_{X}R(x)^{\frac{1}{q_\infty}}\,d\mu(x)\leq\frac{1}{4}\qquad \mbox{and} \qquad  
\int_{X}R(x)^{\frac{1}{p_\infty}}\,d\mu(x)\leq\frac{1}{4}.
\end{equation}

Since $p_\infty\geq p_->1$ it follows from Remark \ref{prelim remark 2} that 
$M^{\D}_{\eta}:L^{p_\infty}(X)\rightarrow L^{q_\infty}(X)$
is a bounded operator.  Thus, there is a constant $C=C(p_\infty,\eta)$ such that
\begin{equation}\label{main subsection 2 equation 7}
\int_X M^{\D}_{\eta} f_2(x)^{q_\infty}\,d\mu(x)\leq
C \left(\int_Xf_{2}(x)^{p_\infty}\,d\mu(x)\right)^{\frac{q_\infty}{ p_\infty}}.
\end{equation}
If we combine (\ref{main subsection 2 equation 5}), 
(\ref{main subsection 2 equation 6}) and (\ref{main subsection 2
  equation 7}),  we get
\begin{equation}\label{main subsection 2 equation 8}
\int_{X\setminus \Omega_{\infty}^{\qq}}(\alpha_2 M^{\D}_{\eta} f_2(x))^{q(x)}\,d\mu(x)\leq
e^{N C_\infty}C \left(\int_X(\alpha_2 f_2(x))^{p_\infty}\,d\mu(x)\right)^{\frac{q_\infty}{ p_\infty}}
+\frac{1}{4}.
\end{equation}

To estimate the integral on the righthand side, we divide it into two
pieces.  Since $0\leq f_2\leq1$, $\rho_{\pp}(f_2)\leq1$ and $1/\pp \in LH$, by
Lemma~\ref{prelim lemma 6} and (\ref{main subsection 2 equation 6}) we
have that
\begin{align}\label{main subsection 2 equation 9}
& \int_{X\setminus \Omega_{\infty}^{\pp}}(\alpha_2
  f_2(x))^{p_\infty}\,d\mu(x) \\
& \qquad =
\alpha_{2}^{p_\infty}\int_{X\setminus
  \Omega_{\infty}^{\pp}}(f_2(x)^{p_\infty
  p(x)})^{\frac{1}{p(x)}}\,d\mu(x) \notag \\
& \qquad \leq
\alpha_{2}^{p_\infty}e^{N C_\infty}\int_{X\setminus \Omega_{\infty}^{\pp}}f_{2}(x)^{p(x)}\,d\mu(x)+
\alpha_{2}^{p_\infty}e^{N C_\infty}\int_{X\setminus
  \Omega_{\infty}^{\pp}}R(x)^{\frac{1}{p_\infty}}\,d\mu(x) \notag \\
& \qquad \leq
\frac{5}{4}\alpha_{2}^{p_\infty}e^{N C_\infty}. \notag
\end{align}

On the other hand, $p_\infty\geq C_{\infty}^{-1} \log(e+d(x,x_0))$ for every
$x\in \Omega_{\infty}^{\pp}$ because $1/\pp\in LH$.
Since $0<\alpha_2<1$ and $0\leq f_2 \leq 1$ we have that
\begin{equation}\label{main subsection 2 equation 10}
\int_{\Omega_{\infty}^{\pp}}(\alpha_2 f_2(x))^{p_\infty}\,d\mu(x)\leq
\int_{X}\alpha_{2}^{C_{\infty}^{-1} \log(e+d(x,x_0))}\,d\mu(x).
\end{equation}

If we combine (\ref{main subsection 2 equation 8}), (\ref{main subsection 2 equation 9}) 
and (\ref{main subsection 2 equation 10}), 
we get
\begin{multline*}
\int_{X\setminus \Omega_{\infty}^{\qq}}(\alpha_2 M^{\D}_{\eta} f_2(x))^{q(x)}\,d\mu(x)\\
\leq
e^{N C_\infty}C
\left(\frac{5}{4} \alpha_{2}^{p_\infty} e^{N C_\infty}+\int_{X}\alpha_{2}^{C_{\infty}^{-1} \log(e+d(x,x_0))}\,d\mu(x)\right)^{q_\infty / p_\infty}+\frac{1}{4}.
\end{multline*}
By Remark \ref{prelim remark 1} the integral on the righthand side 
can be made arbitrarily small by choosing $\alpha_2$ sufficiently
small.  In particular, if we choose  $\alpha_2$ such that
\[ 0<\alpha_2\leq\left(\frac{2}{5 e^{N C_\infty}}\right)^{\frac{1}{p_\infty}}
\left(\frac{1}{4 e^{N C_\infty} C}\right)^{\frac{1}{q_\infty}}
\]
and
\[
\int_{X}\alpha_{2}^{C_{\infty}^{-1} \log(e+d(x,x_0))}\,d\mu(x)\leq
\frac{1}{2} \left(\frac{1}{4 e^{N C_\infty} C}\right)^{\frac{p_\infty}{ q_\infty}},
\]
then we get inequality (\ref{main subsection 2 equation 1}).

\medskip

Finally, suppose  $q_\infty=+\infty$. In this case, by Remark \ref{prelim remark 1} 
we can choose $\alpha_2$ sufficiently small so that 
\[ 
\int_{X\setminus \Omega_{\infty}^{\qq}}(\alpha_2 M^{\D}_{\eta} f_2(x))^{q(x)}\,d\mu(x)\leq
\int_{X}\alpha_{2}^{C_{\infty}^{-1} \log(e+d(x,x_0))}\,d\mu(x)\leq\frac{1}{2},
\]
because $0<\alpha_2<1$, $0 \leq M^{\D}_{\eta} f_2 \leq 1$ and $1/\qq\in LH$.
This completes the proof.

\bigskip

\subsection*{Spaces with finite measure}
We  now consider the case when $\mu(X) < +\infty$. By Lemma
\ref{prelim finite measure}, $\diam(X) < +\infty$.  Hence, by
Corollary \ref{prelim corollary 1}, 
$1/\pp \in LH_0$ implies that $1/\pp \in LH$.  Therefore, by Lemma
\ref{prelim lemma 7} $1/\pp$
satisfies a Diening type estimate. Thus, the proof of inequality 
(\ref{main subsection 1 equation 2}) for $f_1$  carries over without any changes.
The proof of  (\ref{main subsection 1 equation 1}) must be modified to
let us apply Lemma~\ref{prelim lemma 3}.  We sketch the changes, using
the same notation as before. Let
$k_0$ be the smallest integer such that
\[ C^{k_0} > \mu(X)^\eta\avgint_X |f|\,d\mu. \]
Then for all all $k\geq k_0$ we can apply Lemma~\ref{prelim lemma 3}.
Therefore, we can modify the estimate immediately
before~\eqref{eqn:first}, replacing the righthand term by
\[ \sum_{k\geq k_0}\int_{(\Omega_{k}\setminus \Omega_{k+1})\setminus \Omega_{\infty}^{\qq}}(\alpha_1 \beta_1 \gamma_1 C^{k+1})^{q(x)}\,d\mu(x)
+ \int_{ (X\setminus \Omega_\infty^\qq)\setminus \Omega_{k_0}}
(\alpha_1\beta_1\gamma_1 C^{k_0})^{q(x)}\,d\mu(x). \]
The estimate for the first sum proceeds as in the original argument
except that we choose our constants so that it is less than $1/4$. 
To estimate the second integral, choose $\alpha_1<C^{-k_0}$; then it
is bounded by $ \beta_1\gamma_1\mu(X)$,
and by modifying our choices of $\beta_1$ and $\gamma_1$ we can make
this term smaller than $1/4$.

 The estimate for $f_2$ 
can be greatly simplified. Since $M_\eta^\D f_2(x)\leq 1$, for $0 < \alpha_2 < 1$ we get
\begin{multline*}
\int_{X\setminus \Omega_\infty^\qq} (\alpha_2 M_\eta^\D
f_2(x))^{q(x)}\,d\mu(x) +\alpha_2 \|M_\eta^\D
f_2\|_{L^\infty(\Omega_\infty^\qq)} \\
\leq \alpha_2^{q_-} \mu(X) + \alpha_2 \leq \alpha_2 (\mu(X) + 1).
\end{multline*}
Therefore, by choosing $\alpha_2$ sufficiently small we get 
$\rho_\qq(\alpha_2M_\eta^D f_2)\leq 1$. 

\bigskip

\section{Weak type inequalities}
\label{section:weak-type}

In this section we state and prove Theorem~\ref{weak theorem 1}; we
omit those details which are similar to the proof 
of Theorem~\ref{main theorem 1}. As before, we will first
consider the case when $\mu(X)=\infty$, and we will describe the
changes to the proof which are needed when $\mu(X)<\infty$ at the end of the
section.

We begin the proof with the same reductions as before:  it will
suffice to prove inequality~\eqref{weak equation 1} with $M_\eta$
replaced by $M_\eta^\D$ for a fixed dyadic grid $\D$, and  for a function $f$
that is non-negative, bounded and with bounded support, and such that
$\|f\|_\pp=1$.      We will write $f = f_1 + f_2$,  where $f_1 = f \chi_{\{f>1\}}$ and $f_2 = f \chi_{\{f\leq 1\}}$. 
Then for $i = 1 , 2$, the functions $f_i$ satisfy 
$\rho_{\pp}(f_i) \leq \rho_{\pp}(f) \leq \|f\|_\pp = 1$. 
By the definition of the variable Lebesgue norm and because
\[
\underset{t > 0}{\sup}\,t\|\chi_{\{M^{\D}_{\eta} f > t\}}\|_{\qq}\,\leq\,
2\left(\underset{t > 0}{\sup}\,t\|\chi_{\{M^{\D}_{\eta} f_1 > t\}}\|_{\qq} 
+ \underset{t > 0}{\sup}\,t\|\chi_{\{M^{\D}_{\eta} f_2 > t\}}\|_{\qq}\right),
\]
it will suffice to find $\lambda_i>1$ such that
\begin{equation} \label{eqn:modular-eqn}
\underset{t > 0}{\sup}\,\rho_{\qq}(\lambda_{i}^{-1}t\chi_{\{M^{\D}_{\eta} f_i > t\}})\,\leq\,1.
\end{equation}
We will prove each inequality in turn.

\medskip

\subsection*{The estimate for $f_1$}
Let $\lambda_1=\beta_1\gamma_1$, $0<\beta_1,\,\gamma_1<1$.    To prove
the modular estimate for $f_1$ we will prove that for all $t>0$,
\begin{gather}
 \int_{X\setminus \Omega_{\infty}^{\qq}}(\beta_1 \gamma_1t\chi_{\{M^{\D}_{\eta} f_1 > t\}}(x))^{q(x)}
\,d\mu(x) \leq \frac{1}{2} \label{weak main estimate 3}\\
\intertext{and}
 \beta_1 \gamma_1 t \|\chi_{\{M^{\D}_{\eta} f_1 > t\}}\|_{L^{\infty}(\Omega_{\infty}^{\qq})} \leq \frac{1}{2}.
\label{weak main estimate 4}
\end{gather}

We first prove (\ref{weak main estimate 3}).   Assume that $\mu(\{M^{\D}_{\eta} f_1 > t\}) > 0$; 
otherwise there is nothing to prove. Since $f_1$ is bounded and has
bounded support, 
$f_1 \in L^{1}(X)$ and  so by Lemma \ref{prelim lemma 3},
\[ 
 \{M^{\D}_{\eta} f_1 > t\} = \underset{j}{\bigcup}\,Q_j, \] 
where the dyadic cubes $Q_j$ are disjoint and satisfy
\[  t < \mu(Q_j)^{\eta}\avgint_{Q_j}\,f_1(y)\,d\mu(y) \leq C t. \] 
Hence,
\begin{multline} \label{weak estimate 5}
 \int_{X\setminus \Omega_{\infty}^{\qq}}(\beta_1 \gamma_1t
  \chi_{\{M^{\D}_{\eta} f_1 > t\}}(x))^{q(x)}\,d\mu(x) \\
\leq 
\underset{j}{\sum}\int_{Q_{j}\setminus \Omega_{\infty}^{\qq}}
\bigg(\beta_1 \gamma_1\mu(Q_j)^{\eta}\avgint_{Q_j}f_1(y)\,d\mu(y)\bigg)^{q(x)}\,d\mu(x).  
\end{multline}

To estimate the integral inside the parentheses on the righthand side
of (\ref{weak estimate 5}) we will argue much as we did after
inequality~\eqref{eqn:first}.    We will consider the case $0 < \eta <
1$;  the case $\eta = 0$ is essentially the same, bearing in mind that when 
$\eta = 0$, $\pp = \qq$. 

Assume without loss of generality that 
$\mu(Q_j \setminus \Omega_{\infty}^{\qq}) > 0$ for every $j$; otherwise this term contributes
nothing to the sum. When $p_-(Q_j)>1$, by 
Lemma~\ref{prelim lemma 10} and an argument similar to that
between~\eqref{eqn:first} and~(\ref{main subsection 1 equation 11})
it follows that for every $x\in Q_j\setminus \Omega_{\infty}^{\qq}$,
\begin{multline}\label{weak estimate 6}
 \left(\beta_1 \gamma_1\mu(Q_j)^{\eta}\avgint_{Q_j}\,f_1(y)\,d\mu(y)\right)^{q(x)}\\
\leq\,
\left(\beta_1\mu(Q_j)^{-\frac{1}{q_-(Q_j)}}\right)^{q(x)}\,
\left(\gamma_1\int_{Q_j}\,f_{1}(y)^{p(y)}\,d\mu(y)\right)^{\frac{q(x)}{q_-(Q_j)}}.
\end{multline}
Since $1/\qq \in LH$, by 
Corollary~\ref{prelim corollary 1} we may choose $0 < \beta_1 < \min\{1,(1/C)\}$ so that the 
first term on the righthand side of~(\ref{weak estimate 6})
is dominated by $\mu(Q_j)^{-1}$. Since $\supp(f_1)\subset X\setminus \Omega_{\infty}^{\pp}$ up to 
a set of measure zero, the integral of $f_1(\cdot)^{\pp}$ on $Q_j$ is
dominated by $\rho_{\pp}(f_1)\leq 1$.  
Since $0 < \gamma_1 < 1$, the
second term on the righthand side of~(\ref{weak estimate 6})
is dominated by 
\[
\gamma_1\int_{Q_j}\,f_1(y)^{p(y)}\,d\mu(y),
\]
and hence,
\begin{equation}\label{weak estimate 11}
\left(\beta_1 \gamma_1 \mu(Q_j)^{\eta} \avgint_{Q_j}\,f_1(y)\,d\mu(y)\right)^{q(x)}\,\leq\,
\gamma_1 \avgint_{Q_j}\,f_1(y)^{p(y)}\,d\mu(y).
\end{equation}
A similar and simpler argument shows that \eqref{weak estimate 11}
also holds when $p_-(Q_j)=1$.

Therefore, if we combine all of these estimates with~ (\ref{weak estimate 5}), we 
get
\begin{align*}
 \int_{X\setminus \Omega_{\infty}^{\qq}}\,
(\beta_1 &\gamma_1t\chi_{\{M^{\D}_{\eta} f_1 > t\}}(x))^{q(x)}\,d\mu(x)\\
&\notag\leq\,
\underset{j}{\sum}\,\int_{Q_j\setminus \Omega_{\infty}^{\qq}}\,
\left(\gamma_1\avgint_{Q_j}\,f_1(y)^{p(y)}\,d\mu(y)\right)\,d\mu(x)\\
&\notag\leq\,
\underset{j}{\sum}\gamma_1\frac{\mu(Q_j\setminus \Omega_{\infty}^{\qq})}{\mu(Q_j)}
\int_{Q_j}\,f_1(y)^{p(y)}\,d\mu(y)\\
&\notag\leq\,
\gamma_1\int_{X\setminus \Omega_{\infty}^{\qq}}\,f_1(y)^{p(y)}\,d\mu(y)\\
&\notag\leq\,\gamma_1.
\end{align*}
Hence, if we fix $0 < \gamma_1 \leq 1/2$, then inequality
(\ref{weak main estimate 3}) holds with $\beta_1$ and $\gamma_1$
which are independent of $t$.

\medskip

We will now prove inequality (\ref{weak main estimate 4}). 
If 
$t > \|M^{\D}_{\eta} f_1\|_{L^{\infty}(\Omega_{\infty}^{\qq})},$
 then we have that
$\|\chi_{\{M^{\D}_{\eta} f_1 >
  t\}}\|_{L^{\infty}(\Omega_{\infty}^{\qq})} = 0$ and so there is nothing 
to prove.  On the other hand, if we fix $t \leq \|M^{\D}_{\eta}
f_1\|_{L^{\infty}(\Omega_{\infty}^{\qq})}$,  then 
\[
t\,\beta_1\,\gamma_1\,\|\chi_{\{M^{\D}_{\eta} f_1 > t\}}\|_{L^{\infty}(\Omega_{\infty}^{\qq})} \leq 
\beta_1\,\gamma_1\,\|M^{\D}_{\eta} f_1\|_{L^{\infty}(\Omega_{\infty}^{\qq})}.
\]

The argument we used to prove inequality (\ref{main subsection 1 equation 2}) did not depend 
on the fact that $q_- > 1$ and continues to hold with minor modifications
when $q_- = 1$. Thus, we can find constants $\beta_1$ and $\gamma_1$ such that inequality (\ref{weak main estimate 4})
holds for all $t$. 

\bigskip

\subsection*{The estimate for $f_2$}
To prove the modular estimate~\eqref{eqn:modular-eqn}, let
$\lambda_2^{-1}=\alpha_2$, $0<\alpha_2<1$.  We will show that we can
find $\alpha_2$ sufficiently small so that for all $t$, 
\begin{gather}
\int_{X\setminus \Omega_{\infty}^{\qq}}\,
(\alpha_2t\chi_{\{M^{\D}_{\eta} f_2 > t\}}(x))^{q(x)}\,d\mu(x) \leq 
\frac{1}{2}
 \label{weak main estimate 5}\\
\intertext{and}
\alpha_2t\|\chi_{\{M^{\D}_{\eta} f_2 > t\}}\|_{L^{\infty}(\Omega_{\infty}^{\qq})} \leq 
\frac{1}{2}.\label{weak main estimate 6}
\end{gather}

We begin with two observations.  First, 
in Section \ref{main subsection 2} we showed that $M^{\D}_{\eta}
f_2(\cdot) \leq 1$, and the proof did not rely on the fact that $p_- > 1$. Second,
since $f_2$ is bounded and has bounded support, $f_2 \in L^{p}(X)$ for
all $p>1$, and so by the weak $(p,q)$ inequality for $M_\eta^\D$ (see Remark~\ref{prelim remark 2}),
$\mu(\{M^{\D}_{\eta} f_2 > t\}) < +\infty$ for all $t > 0$.

We first prove inequality (\ref{weak main estimate 5}). There are two cases:
$q_\infty < +\infty$ and $q_\infty = +\infty$. First assume that 
$q_\infty < +\infty$.   Since $M^{\D}_{\eta} f_2(\cdot) \leq 1$, we
may assume $t < 1$ since otherwise there is nothing to prove.  Then by
Lemma \ref{prelim lemma 6} we have that
\begin{align}
&\int_{X\setminus \Omega_{\infty}^{\qq}}\,
(\alpha_2t\chi_{\{M^{\D}_{\eta} f_2 > t\}}(x))^{q(x)}\,d\mu(x) \label{weak estimate 9}\\
&\notag \qquad \leq
e^{N C_{\infty}}\int_{X\setminus \Omega_{\infty}^{\qq}}\,
(\alpha_2t\chi_{\{M^{\D}_{\eta} f_2 > t\}}(x))^{q_\infty}\,d\mu(x)\,+\,
\int_{X}\,R(x)^{\frac{1}{q_\infty}}\,d\mu(x)\\
&\notag\qquad \leq
e^{N C_\infty}\,\alpha_{2}^{q_\infty}\,t^{q_\infty}\,\mu(\{M^{\D}_{\eta} f_2 > t\})\,+\,
\int_{X}\,R(x)^{\frac{1}{q_\infty}}\,d\mu(x),
\end{align}
where $R(x) = (e + d(x , x_0))^{-N}$. By Lemma \ref{prelim lemma 5} and the dominated convergence 
theorem we can choose $N$ large enough so that 
\begin{equation}\label{weak estimate 10}
\int_{X}\,R(x)^{\frac{1}{q_\infty}}\,d\mu(x) \leq \frac{1}{4}\qquad 
\mbox{and}\qquad
\int_{X}\,R(x)^{\frac{1}{p_\infty}}\,d\mu(x) \leq \frac{1}{4}.
\end{equation}
Again by Remark \ref{prelim remark 2} we know that $M^{\D}_{\eta}
:L^{p_\infty}(X)\rightarrow L^{q_\infty,\infty}(X)$, If we combine
this fact with inequalities~(\ref{weak estimate 9}) and~(\ref{weak estimate 10}), and then repeat the
argument used to prove  (\ref{main subsection 2 equation 9}) and
(\ref{main subsection 2 equation 10}) (which holds without change), we
get
\begin{align*}
& \int_{X\setminus \Omega_{\infty}^{\qq}}\,
(\alpha_2t\chi_{\{M^{\D}_{\eta} f_2>t\}}(x))^{q(x)}\,d\mu(x) \\
& \qquad \leq
e^{N C_\infty}\,C(p_\infty, X)\,
\left(\int_{X}\,(\alpha_2\,f_2(y))^{p_\infty}\,d\mu(y)\right)^{\frac{q_\infty}{p_\infty}}\,+\,
\frac{1}{4} \\
& \qquad \leq
e^{N C_\infty}C(p_\infty,X)
\left(\frac{5}{4} \alpha_{2}^{p_\infty} e^{N C_\infty}+\int_{X}\alpha_{2}^{C_{\infty}^{-1} \log(e+d(x,x_0))}\,d\mu(x)\right)^{\frac{q_\infty}{ p_\infty}}+\frac{1}{4}.
\end{align*}

By Remark \ref{prelim remark 1} the integral on the righthand side 
can be made arbitrarily small by choosing $\alpha_2$ sufficiently
small. Therefore, we can choose  $0 < \alpha_2 < 1$ such that
inequality (\ref{weak main estimate 5}) holds for all $t$.

Now suppose $q_\infty = +\infty$. In this case, since $1/\qq \in LH$, $0<t<1$ and  
$0\leq \chi_{\{M^{\D}_{\eta} f_2>t\}}(\cdot)\leq1$, by Remark \ref{prelim remark 1} 
we can choose $0 < \alpha_2 < 1$ sufficiently small so that for all $t>0$,
\[ 
\int_{X\setminus \Omega_{\infty}^{\qq}}(\alpha_2\,t\,\chi_{\{M^{\D}_{\eta} f_2 > t\}}(x))^{q(x)}\,d\mu(x)\leq
\int_{X}\alpha_{2}^{C_{\infty}^{-1} \log(e+d(x,x_0))}\,d\mu(x)\leq\frac{1}{2}.
\]

\medskip

We now prove (\ref{weak main estimate 6}).   If $t \geq 1$ there is
nothing to prove since the lefthand side is 0.  On the other hand, 
if $t < 1$, then by choosing $0 < \alpha_2 \leq 1/2$ we get 
\[ \alpha_2\,t\,\|\chi_{\{M^{\D}_{\eta} f_2 >
  t\}}\|_{L^{\infty}(\Omega_{\infty}^{\qq})} \leq \alpha_2 \leq
1/2, \]
and this establishes inequality (\ref{weak main estimate 6}) for all 
$t>0$.  This completes the estimate for $f_2$ and also the proof of
Theorem~\ref{weak theorem 1}.  

\bigskip

\subsection*{Spaces with finite measure}

Finally we consider the case when $\mu(X) < +\infty$. By Lemma
\ref{prelim finite measure} and Corollary~\ref{prelim corollary 1},
$1/\pp \in LH$, and so by Lemma \ref{prelim lemma 7} $1/\pp$ satisfies
a Diening type condition. Therefore, the estimate for $f_1$ is
unchanged when $t>\mu(X)^\eta \avgint_X |f|\,d\mu$.  When $t$ is
smaller than this bound we choose $\beta_1$ accordingly and bound the
lefthand side of~\eqref{weak main estimate 3} by $\gamma_1\mu(X)$; by
the appropriate choice of $\gamma_1$ this can be made as small as desired.

 The estimate for $f_2$ can be
greatly simplified.  Since $\mu(\{M^{\D}_{\eta} f_2 > t\}) > 0$ only
when $t < 1$,  we have that
\begin{multline*}
\int_{X\setminus \Omega_{\infty}^{\qq}}
(\alpha_2\,t\,\chi_{\{M^{\D}_{\eta} f_2 > t\}}(x))^{q(x)}\,d\mu(x)\,
+\,\alpha_2\,t\,\|\chi_{\{M^{\D}_{\eta} f_2 > t\}}\|_{L^{\infty}(\Omega_{\infty}^{\qq})}\\
\leq\,\alpha_{2}^{q_-}\,\mu(X)\,+\,\alpha_2\,
\leq\,\alpha_2(\mu(X) + 1),
\end{multline*}
if $0<\alpha_2<1$.
Hence, if we choose $\alpha_2$ sufficiently small, we get the desired inequality.

\section{Fractional integral operators}
\label{section:fractional}

In this section we prove Theorems~\ref{frac int theorem 1}
and~\ref{frac int theorem 2}. This requires a pointwise estimate
that relates the fractional integral operator to the fractional maximal
operator. We prove this estimate in Proposition~\ref{Welland}.  Given this
inequality, the actual proofs follow exactly as they do in the
Euclidean case, and we refer the reader to \cite{MR2414490} for
complete details.

\begin{prop}\label{Welland}
Let $(X,d,\mu)$ be a reverse doubling space. Given $0 < \eta < 1$,  fix $\varepsilon$, 
$0 < \varepsilon < \min \{\eta , 1-\eta\}$. Then there exists a constant $C=C(\eta , \varepsilon,X)$ such that
for every $f \in L^{1}_{loc}(X)$ and for every $x \in X$,
\[
|I_{\eta} f(x)| \leq C\,(M_{\eta-\varepsilon} f(x) M_{\eta+\varepsilon} f(x))^{\frac{1}{2}}.
\]
\end{prop}

\begin{remark} 
It is only in the proof of this result that we use the assumption that $\mu$ is a reverse
doubling measure. We are not certain if Proposition~\ref{Welland} remains true
without this hypothesis.
\end{remark}

\begin{proof}
We first assume that $\mu(X) = +\infty$; we will describe the changes
to the proof when $\mu(X)<+\infty$ afterwards. Given 
$f \in L^{1}_{loc}(X)$ and $x \in X$ define

\begin{gather*}
 I_1 f(x)\,=\,\int_{B(x,\delta)}\,\frac{f(y)}{\mu(B(x,d(x,y)))^{1-\eta}}\,d\mu(y) \\
 I_2 f(x)\,=\,\int_{X\setminus B(x,\delta)}\,\frac{f(y)}{\mu(B(x,d(x,y)))^{1-\eta}}\,d\mu(y),
\end{gather*}
where the precise value of $\delta>0$ will be fixed below.  Clearly
$|I_\eta f(x)|\leq |I_1f(x)|+|I_2f(x)|$ and so we will estimate each term
on the righthand side separately.  

\medskip

We first estimate $|I_1 f(x)|$. For $i \geq 0$ define 
\[ R_i=\{y\in X\,:\,2^{-i-1}\delta \leq d(x,y) <
2^{-i}\delta\}. \]
Since the measure $\mu$ is both doubling and reverse doubling, we have that
\begin{align}
& \notag |I_1 f(x)| \\
& \notag \qquad \leq
\sum_{i\geq 0}
  \int_{R_i}\frac{|f(y)|}{\mu(B(x,d(x,y)))^{1-\eta}}\,d\mu(y) \\
&\notag  \qquad \leq
\underset{i\geq 0}{\sum}\,
\left(\frac{\mu(B(x,2^{-i}\delta))}{\mu(B(x,2^{-i-1}\delta))}\right)^{1-\eta}\,
\frac{\mu(B(x,2^{-i}\delta))^{\varepsilon}}{\mu(B(x,2^{-i}\delta))^{1-\eta+\varepsilon}}\,
\int_{B(x,2^{-i}\delta)}\,|f(y)|\,d\mu(y)\\
&\notag \qquad \leq
\underset{i\geq 0}{\sum} (C C_\mu)^{1-\eta}\,(\gamma^{\varepsilon})^{i}\,\mu(B(x,\delta))^{\varepsilon}\,
M_{\eta-\varepsilon} f(x) \\
& \qquad \leq
\frac{(C C_\mu)^{1-\eta}}{1-\gamma^{\varepsilon}}\,\mu(B(x,\delta))^{\varepsilon}\,
M_{\eta-\varepsilon} f(x),\label{frac int I 1 estimate 2}
\end{align}
where the constant $C$ comes from the lower mass bound which is satisfied by a doubling measure.
The last inequality holds because $0<\gamma<1$ ensures that the geometric
series converges.

\medskip

We now estimate $|I_2 f(x)|$  in a similar fashion. For $i \geq 1$ we
now define
\[ R_i=\{y\in X : 2^{i-1}\delta \leq d(y,x) < 2^{i}\delta\}; \]
then essentially the same argument shows that
\begin{align}
& |I_2 f(x)| \notag \\
& \notag \qquad \leq
\underset{i\geq 1}{\sum}
\left(\frac{\mu(B(x,2^{i}\delta))}{\mu(B(x,2^{i-1}\delta))}\right)^{1-\eta}\,
\frac{\mu(B(x,2^{i}\delta))^{-\varepsilon}}{\mu(B(x,2^{i}\delta))^{1-\eta-\varepsilon}}\,
\int_{B(x,2^{i}\delta)}\,|f(y)|\,d\mu(y) \\
&\notag \qquad \leq
\underset{i\geq 0}{\sum}\,
(C C_\mu)^{1-\eta}\,\mu(B(x,\delta))^{-\varepsilon}\,M_{\eta + \varepsilon} f(x)\,
(\gamma^{\varepsilon})^{i} \\
&\qquad \leq \frac{(C C_\mu)^{1-\eta}}{1-\gamma^{\varepsilon}}\,\mu(B(x,\delta))^{-\varepsilon}\,
M_{\eta + \varepsilon} f(x).\label{frac int I 2 estimate 2}                
\end{align} 

\medskip

To complete the proof we will show that there exists $\delta>0$ such
that
\begin{equation}\label{frac int vol estimate 1}
\frac{1}{C_\mu}\,
\left(\frac{M_{\eta + \varepsilon} f(x)}{M_{\eta-\varepsilon} f(x)}\right)^{\frac{1}{2 \varepsilon}}
<
\mu(B(x,\delta)) \leq 
\left(\frac{M_{\eta + \varepsilon} f(x)}{M_{\eta-\varepsilon} f(x)}\right)^{\frac{1}{2 \varepsilon}}.
\end{equation} 
If we assume this for the moment and 
substitute (\ref{frac int vol estimate 1}) into inequalities (\ref{frac int I 1 estimate 2}) and
(\ref{frac int I 2 estimate 2}), we get

\[
|I_\eta f(x)| \leq C\,(M_{\eta+\varepsilon} f(x) M_{\eta-\varepsilon} f(x))^{\frac{1}{2}},
\]
which completes the proof when $\mu(X) = +\infty$.

\medskip

To prove~\eqref{frac int vol estimate 1} define the set
\[
\left\{r>0 :\mu(B(x,r)) \leq \left(\frac{M_{\eta + \varepsilon}
      f(x)}{M_{\eta-\varepsilon} f(x)}\right)^{\frac{1}{2
      \varepsilon}}\right\}.
\]
If this set were empty,  then because $\mu$ is non-trivial this would imply that
$\mu(B(x,r))$ is greater than the ratio on the righthand side. By choosing 
a sequence of radii such that $r_i>0$, $r_i>r_{i+1}$ and $r_i\rightarrow 0$ as $i\rightarrow \infty$, we
would get that $\mu(\{x\})>0$ and this contradicts Lemma \ref{lemma:no-atoms}. Therefore, the 
given set is non-empty. Furthermore, since we can exclude the trivial cases where either 
$f=0$ a.e. or $|f|=+\infty$ on a set of positive measure, it follows that the given set is 
bounded above.

Hence,
\[
\delta_0 =
\sup\bigg\{r\,>\,0\,:\,\mu(B(x,r)) \leq 
\left(\frac{M_{\eta + \varepsilon} f(x)}{M_{\eta-\varepsilon} f(x)}\right)^{\frac{1}{2 \varepsilon}}\bigg\}
\]
is well-defined.
Fix $\delta$ such that $\delta_0/2 < \delta < \delta_0$; then
\[ \mu(B(x,\delta)) \leq \left(\frac{M_{\eta + \varepsilon}
    f(x)}{M_{\eta-\varepsilon} f(x)}\right)^{\frac{1}{2 \varepsilon}}
< \mu(B(x,2\delta)) \leq C_\mu(B(x,\delta)), \]
and inequality~\eqref{frac int vol estimate 1} follows at once.

\bigskip

Now assume that $\mu(X) < +\infty$;  we again ignore the trivial cases where $f=0$ a.e. or 
$|f|=+\infty$ on a set of positive measure. 
Let $x\in X$ and let $B$ be an arbitrary ball which contains $x$; then
we have that
\[ \mu(B)^{\eta+\varepsilon} \avgint_{B}\,|f|\,d\mu \leq 
\mu(X)^{2\varepsilon} \mu(B)^{\eta-\varepsilon} \avgint_{B}\,|f|\,d\mu
\leq \mu(X)^{2\varepsilon} M_{\eta-\varepsilon} f(x). \]
If we take the supremum over all such balls, we get
\[
\left(\frac{M_{\eta+\varepsilon}f(x)}{M_{\eta-\varepsilon}f(x)}\right)^{\frac{1}{2\varepsilon}}
\leq \mu(X) < +\infty.
\]

Define the set 

\[
\left\{r>0 : \mu(B(x,r)) \leq \frac{1}{2}
\left(\frac{M_{\eta+\varepsilon}f(x)}{M_{\eta-\varepsilon}f(x)}\right)^{\frac{1}{2\varepsilon}}\right\}.
\]

If this set were empty, then $\mu(B(x,r))$ would be greater than the ratio on the righthand side for 
every $r>0$. By choosing radii $r_i$ such that $B(x,r_1)\subsetneq X$, $r_i>0$, 
$r_i>r_{i+1}$ and $r_i\rightarrow 0$ as $i\rightarrow \infty$ we get that $\mu(\{x\})>0$ which
contradicts Lemma~\ref{lemma:no-atoms}. Therefore, the set is non-empty and bounded above and 
its supremum, which we denote by $\delta_0$, is well-defined. Thus, for $0<\delta<\delta_0<2\delta$ we 
have

\[
\frac{1}{2C_\mu} 
\left(\frac{M_{\eta+\varepsilon}f(x)}{M_{\eta-\varepsilon}f(x)}\right)^{\frac{1}{2\varepsilon}}
\leq \mu(B(x,\delta))\leq
\frac{1}{2}
\left(\frac{M_{\eta+\varepsilon}f(x)}{M_{\eta-\varepsilon}f(x)}\right)^{\frac{1}{2\varepsilon}}.
\]

The upper bound on $\mu(B(x,\delta))$ ensures that $B(x,\delta)\subsetneq X$ and we may apply the 
reverse doubling condition to the balls $B(x,2^{-i}\delta)$ for $i\geq 0$. Thus the estimate for 
$I_1 f(x)$ remains unchanged. 

However, we need to be more careful with the estimate for 
$I_2 f(x)$ because this involves positive dilations of the ball $B(x,\delta)$ and the reverse doubling
condition only applies to balls which are strictly contained in $X$.
Let $k$ be the smallest integer such that
$2^{k}\delta>\diam(X)$.  If we define $R_i$ as before, then for $i>k$,
$R_i=\emptyset$.  Hence, estimating as before, we have that
\[ |I_2f(x)| \leq \sum_{i=1}^k 
\left(\frac{\mu(B(x,2^{i}\delta))}{\mu(B(x,2^{i-1}\delta))}\right)^{1-\eta}\,
\frac{\mu(B(x,2^{i}\delta))^{-\varepsilon}}{\mu(B(x,2^{i}\delta))^{1-\eta-\varepsilon}}\,
\int_{B(x,2^{i}\delta)}\,|f(y)|\,d\mu(y) \]

Let $a$ be the smallest integer such that $A_0\leq 2^a$.  Then for
$1\leq i \leq k$,
\[ 2^{i-a-4}\delta \leq (8A_0)^{-1} 2^{k-1}\delta \leq
(8A_0)^{-1}\diam(X). \]
In the proof of Lemma \ref{lemma:no-atoms} we showed that balls with radii strictly less than
$(8A_0)^{-1}\diam(X)$ are strictly contained in $X$,  and so the reverse doubling condition 
holds for the balls $B(x,2^{i-a-4}\delta)$. Hence,
\begin{multline*}
 \mu(B(x,2^i\delta))^{-\varepsilon} 
\leq \mu(B(x,2^{i-a-4}\delta))^{-\varepsilon} \\
\leq (\gamma^\varepsilon)^i \mu(B(x,2^{-a-4}\delta))^{-\varepsilon}
\leq C_\mu^{(a+4)\varepsilon} (\gamma^\varepsilon)^i \mu(B(x,\delta))^{-\varepsilon}. 
\end{multline*}
If we substitute this into the estimate for $I_2 f(x)$, we get
\[ |I_2f(x)|
 \leq \frac{C^{1-\eta}C_\mu^{(a+4)\varepsilon+1-\eta}}{1-\gamma^{\varepsilon}}\,\mu(B(x,\delta))^{-\varepsilon}\,
M_{\eta + \varepsilon} f(x), \]
where $C$ is the constant from the lower mass bound which is satisfied by a doubling measure. We may now
use the lower bound for $\mu(B(x,\delta))$ in order to complete the proof of Proposition
\ref{Welland} when $\mu(X)<+\infty$.

\end{proof}

\bibliographystyle{plain}
\bibliography{dcu_ps_2}

\def\cprime{$'$}
\begin{thebibliography}{10}

\bibitem{MR3322604}
T.~Adamowicz, P.~Harjulehto, and P.~H{\"a}st{\"o}.
\newblock Maximal operator in variable exponent {L}ebesgue spaces on unbounded
  quasimetric measure spaces.
\newblock {\em Math. Scand.}, 116(1):5--22, 2015.

\bibitem{MR2516251}
A.~Almeida and S.~Samko.
\newblock Fractional and hypersingular operators in variable exponent spaces on
  metric measure spaces.
\newblock {\em Mediterr. J. Math.}, 6(2):215--232, 2009.

\bibitem{MR3302574}
T.~C. Anderson, D.~Cruz-Uribe, and K.~Moen.
\newblock Logarithmic bump conditions for {C}alder\'on-{Z}ygmund operators on
  spaces of homogeneous type.
\newblock {\em Publ. Mat.}, 59(1):17--43, 2015.

\bibitem{MR2867756}
A.~Bj{\"o}rn and J.~Bj{\"o}rn.
\newblock {\em Nonlinear potential theory on metric spaces}, volume~17 of {\em
  EMS Tracts in Mathematics}.
\newblock European Mathematical Society (EMS), Z\"urich, 2011.

\bibitem{MR1430157}
M.~Bramanti and M.~C. Cerutti.
\newblock Commutators of singular integrals on homogeneous spaces.
\newblock {\em Boll. Un. Mat. Ital. B.(7)}, 10(4):843--883, 1996.

\bibitem{MR2414490}
C.~Capone, D.~Cruz-Uribe, and A.~Fiorenza.
\newblock The fractional maximal operator and fractional integrals on variable
  {$L^p$} spaces.
\newblock {\em Rev. Mat. Iberoam.}, 23(3):743--770, 2007.

\bibitem{MR1096400}
M.~Christ.
\newblock A {$T(b)$} theorem with remarks on analytic capacity and the {C}auchy
  integral.
\newblock {\em Colloq. Math.}, 60/61(2):601--628, 1990.

\bibitem{MR2493649}
D.~Cruz-Uribe, L.~Diening, and A.~Fiorenza.
\newblock A new proof of the boundedness of maximal operators on variable
  {L}ebesgue spaces.
\newblock {\em Boll. Unione Mat. Ital. (9)}, 2(1):151--173, 2009.

\bibitem{MR3026953}
D.~Cruz-Uribe and A.~Fiorenza.
\newblock {\em Variable {L}ebesgue spaces}.
\newblock Applied and Numerical Harmonic Analysis. Birkh\"auser/Springer,
  Heidelberg, 2013.
\newblock Foundations and harmonic analysis.

\bibitem{MR2385658}
G.~Di~Fazio, C.~E. Guti{\'e}rrez, and E.~Lanconelli.
\newblock Covering theorems, inequalities on metric spaces and applications to
  {PDE}'s.
\newblock {\em Math. Ann.}, 341(2):255--291, 2008.

\bibitem{DHHR}
L.~Diening, P.~Harjulehto, P.~H\"{a}st\"{o}, and M.~Ruzicka.
\newblock {\em Lebesgue and Sobolev spaces with variable exponents}.
\newblock Lecture Notes in Mathematics. Springer, 2011.

\bibitem{MR2248828}
T.~Futamura, Y.~Mizuta, and T.~Shimomura.
\newblock Sobolev embeddings for variable exponent {R}iesz potentials on metric
  spaces.
\newblock {\em Ann. Acad. Sci. Fenn. Math.}, 31(2):495--522, 2006.

\bibitem{garcia-cuerva-martell01}
J.~Garc{\'{\i}}a-Cuerva and J.~M Martell.
\newblock Two-weight norm inequalities for maximal operators and fractional
  integrals on non-homogeneous spaces.
\newblock {\em Indiana Univ. Math. J.}, 50(3):1241--1280, 2001.

\bibitem{MR2738962}
O.~Gorosito, G.~Pradolini, and O.~Salinas.
\newblock Boundedness of fractional operators in weighted variable exponent
  spaces with non doubling measures.
\newblock {\em Czechoslovak Math. J.}, 60(135)(4):1007--1023, 2010.

\bibitem{MR3183648}
L.~Grafakos, L.~Liu, D.~Maldonado, and D.~Yang.
\newblock Multilinear analysis on metric spaces.
\newblock {\em Dissertationes Math. (Rozprawy Mat.)}, 497:121, 2014.

\bibitem{MR1336257}
P.~Haj{\l}asz and P.~Koskela.
\newblock Sobolev meets {P}oincar\'e.
\newblock {\em C. R. Acad. Sci. Paris S\'er. I Math.}, 320(10):1211--1215,
  1995.

\bibitem{HHP}
P.~Harjulehto, P.~H{\"a}st{\"o}, and M.~Pere.
\newblock Variable exponent lebesgue spaces on metric spaces: The
  hardy--littlewood maximal operator.
\newblock {\em Real Anal. Exchange}, 30:87--104, 2004/2005.

\bibitem{MR2296640}
P.~Harjulehto, P.~H{\"a}st{\"o}, and M.~Pere.
\newblock Variable exponent {S}obolev spaces on metric measure spaces.
\newblock {\em Funct. Approx. Comment. Math.}, 36:79--94, 2006.

\bibitem{MR1800917}
J.~Heinonen.
\newblock {\em Lectures on analysis on metric spaces}.
\newblock Universitext. Springer-Verlag, New York, 2001.

\bibitem{MR2901199}
T.~Hyt{\"o}nen and A.~Kairema.
\newblock Systems of dyadic cubes in a doubling metric space.
\newblock {\em Colloq. Math.}, 126(1):1--33, 2012.

\bibitem{MR3058926}
A.~Kairema.
\newblock Two-weight norm inequalities for potential type and maximal operators
  in a metric space.
\newblock {\em Publ. Mat.}, 57(1):3--56, 2013.

\bibitem{MR3100969}
V.~Kokilashvili, A.~Meskhi, and M.~Sarwar.
\newblock Two-weight norm estimates for maximal and {C}alder\'on-{Z}ygmund
  operators in variable exponent {L}ebesgue spaces.
\newblock {\em Georgian Math. J.}, 20(3):547--572, 2013.

\bibitem{MR2387413}
V.~Kokilashvili and S.~Samko.
\newblock The maximal operator in weighted variable spaces on metric measure
  spaces.
\newblock {\em Proc. A. Razmadze Math. Inst.}, 144:137--144, 2007.

\bibitem{MR2569546}
T.~M{\"a}k{\"a}l{\"a}inen.
\newblock Adams inequality on metric measure spaces.
\newblock {\em Rev. Mat. Iberoam.}, 25(2):533--558, 2009.

\bibitem{sawyer-wheeden92}
E.~T. Sawyer and R.~L. Wheeden.
\newblock Weighted inequalities for fractional integrals on {E}uclidean and
  homogeneous spaces.
\newblock {\em Amer. J. Math.}, 114(4):813--874, 1992.

\bibitem{MR0369641}
G.~V. Welland.
\newblock Weighted norm inequalities for fractional integrals.
\newblock {\em Proc. Amer. Math. Soc.}, 51:143--148, 1975.

\end{thebibliography}

\end{document}